\documentclass[11pt]{article}
\usepackage[pagebackref,letterpaper=true,colorlinks=true,pdfpagemode=none,urlcolor=blue,linkcolor=blue,citecolor=blue,pdfstartview=FitH]{hyperref}
\usepackage{amsmath,amsfonts,amsthm,mathrsfs}
\usepackage{amssymb}
\usepackage[margin=1in]{geometry}
\usepackage{graphicx}
\usepackage{color}
\usepackage{enumitem}

\allowdisplaybreaks

\def\P{{\mathbb P}}
\def\E{{\mathbb E}}

\def\N{{\mathbb N}}

\def\R{{\mathbb R}}

\def\Q{{\mathbb Q}}
\def\F{{\mathcal F}}
\def\eps{{\epsilon}}

\makeatletter \@addtoreset{equation}{section}

\newtheorem{theorem}{Theorem}[section]
\newtheorem{lemma}[theorem]{Lemma}
\newtheorem{definition}[theorem]{Definition}
\newtheorem{corollary}[theorem]{Corollary}
\newtheorem{proposition}[theorem]{Proposition}
\theoremstyle{definition}

\newtheorem{example}[theorem]{Example}
\newtheorem{remark}[theorem]{Remark}

\title{The Stochastic Solution to a Cauchy Problem for Degenerate Parabolic Equations 
\thanks{This research was supported in part by National Natural Science Foundation of China (No. 11526090 and No. 11601163),  
NSF of Guangdong Province of China(No. 2016A030313448),
the University of Colorado at Boulder (11003573), 
a grant from the Research Growth Initiative of UW-Milwaukee, 
RGC of CityU 109613, and the CityU SRG 7004241.
}
}

\author{
Xiaoshan Chen
\thanks{School of Mathematical Science, South China Normal University, Guangzhou 510631, China, xchen53@gmail.com.}
\and
Yu-Jui Huang
\thanks{Department of Applied Mathematics, University of Colorado, Boulder, CO 80303, USA, yujui.huang@colorado.edu.}
\and
Qingshuo Song
\thanks{Department of Mathematics, City University of Hong Kong, 83 Tat Chee Avenue, Kowloon Tong, Hong Kong, song.qingshuo@cityu.edu.hk.}
\and Chao Zhu
\thanks{Department of Mathematical Sciences, University of Wisconsin-Milwaukee, Milwaukee, WI 53201, USA, zhu@uwm.edu.}
}

\begin{document}
\maketitle

\begin{abstract}
We study the stochastic solution to a Cauchy problem for a degenerate parabolic equation arising from option pricing. When the diffusion coefficient of the underlying price process is locally H\"older continuous with exponent $\delta\in (0, 1]$, the stochastic solution, which represents the price of a European option, is shown to be a classical solution to the Cauchy problem. This improves the standard requirement $\delta\ge 1/2$. Uniqueness results, including a Feynman-Kac formula and a comparison theorem, are established without assuming the usual linear growth condition on the diffusion coefficient. When the stochastic solution is not smooth, it is characterized as the limit of an approximating smooth stochastic solutions. In deriving the main results, we discover a new, probabilistic proof of Kotani's criterion for martingality of a one-dimensional diffusion in natural scale.



\end{abstract}

\section{Introduction}
We study the one-dimensional diffusion in natural scale
\begin{equation}\label{SDE X}
 d X^{t,x}_{s} = \sigma(X^{t,x}_{s}) dW_{s},\quad  X^{t,x}_{t} = x\in I=(0,\infty),
\end{equation}
where $W$ is a standard Brownian motion and $\sigma:I \mapsto \R$ is a Borel function satisfying the standing assumption
\begin{equation}\label{standing asm'}
\sigma(\cdot)\neq 0\ \hbox{and}\ \sigma^{-2}(\cdot)\ \hbox{is locally integrable}\ \hbox{on $I$}.
\end{equation}
We also assume that $X$ is absorbed at $0$, once it arrives there.  

Under the paradigm of no-arbitrage pricing, one usually postulates a diffusion model as \eqref{SDE X} for the price process $X$ of the underlying asset, under a risk-neutral measure calibrated to market data. Given a payoff function $g:[0,\infty)\mapsto\R$ and a terminal time $T>0$, the price of the corresponding European option is formulated as 
\begin{equation}\label{defn U}
U(t,x):=\E^{t,x}[g(X^{t,x}_T)],
\end{equation}  
where $\E^{t,x}$ denotes the expectation taken under a law $\P^{t,x}$ such that $X^{t,x}_{t} = x$ a.s. A heuristic use of It\^{o}'s rule shows that if $U$ is smooth, it is a {\it classical solution} to the Cauchy problem
\begin{equation}\label{PDE X g}
\begin{cases}
\partial_t u+\frac{1}{2}\sigma^2(x)\partial_{xx}u=0,\ \ &(t,x)\in[0,T)\times (0,\infty);\\
u(T,x)=g(x),\ \ &x\in(0,\infty);\\
u(t,0)=g(0),\ \ &t\in[0,T].
\end{cases}
\end{equation}
It is, however, difficult to establish the smoothness of $U$ a priori: standard results of parabolic equations (see e.g. \cite[Chapter 6]{Friedman-book-06}) cannot be applied here, as $\sigma$ may (i) grow faster than linearly and (ii) degenerate on the boundary of the state space $I$. By contrast, under fairly general conditions, the strong Markov property of $X$ readily implies that $U$ is the unique {\it stochastic solution} to \eqref{PDE X g}. The notion of stochastic solution was introduced by Stroock \& Varadhan \cite{SV72}, and further developed in the literature of stochastic analysis and mathematical finance, see e.g. \cite{BKX12, BS12, ET09, JT06} and the references therein. 

The goal of this paper is to characterize the stochastic solution $U$, from both probabilistic and analytic perspectives. We focus on the following questions:
\begin{itemize}
\item[(Q1)] What condition gives a Feynman-Kac formula for \eqref{PDE X g} (i.e. if $u$ is a classical solution to \eqref{PDE X g}, then $u$ must coincide with $U$)?
\item [(Q2)] What condition guarantees that $U$ is a classical solution to \eqref{PDE X g})?
\item [(Q3)] If $U$ is not smooth, how do we characterize $U$?
\end{itemize}

As a preparatory step, we study the martingality of $X$. Our contribution is a new proof of Kotani's criterion for $X$ being a martingale (Theorem~\ref{thm:X true mart}). Kotani's original proof in \cite{Kotani06}, together with the detailed exposition in \cite{HP08}, is based on analytic methods. We provide a probabilistic proof, which is more concise than the arguments in \cite{Kotani06, HP08}. Note that Kotani's criterion deals with a general state space $I=(\ell,r)$ for $X$, with $-\infty\le \ell< r\le\infty$. For the special case $I=(0,\infty)$, the criterion reduces to: $X$ is a martingale if and only if 
\begin{equation}\label{DS02}
 \int_{1}^{\infty} \frac{x}{\sigma^{2}(x)} dx = \infty.
\end{equation}
This result was proved under several different approaches in \cite{DS02, MU12, Karatzas-Ruf-13}; see Remark~\ref{rem:literature} for details.  

Our strategy for (Q1) is to leverage on the connection between {\it local} stochastic solutions (Definition \ref{defn:stoch sol V}) and  martingality of $X$. We identify that ``$U$ is the unique local stochastic solution to \eqref{PDE X g}'' if and only if ``$X$ is a martingale'', i.e. ``\eqref{DS02} holds'' (Proposition~\ref{prop:equiv}). Since a classical solution is necessarily a local stochastic solution as mentioned in Remark~\ref{p:class-stoch}, a Feynman-Kac formula for \eqref{PDE X g} holds under \eqref{DS02} (Corollary~\ref{coro:FK}). This result is non-standard, compared to \cite[Chapter 6]{Friedman-book-06} and \cite[Section 5.7]{KS-book-91}, in that under \eqref{DS02} $\sigma$ may grow faster than linearly and there is no continuity assumption on $\sigma$. Moreover, this Feynman-Kac formula generalizes Bayraktar \& Xing \cite[Theorem 1]{BX10}; see Remark~\ref{rem:generalize Thm1}.

Ekstr\"{o}m \& Tysk \cite[Section 3]{ET09} studied (Q2), and showed that local H\"{o}lder continuity of $\sigma$ with exponent $\delta\ge 1/2$ is sufficient for $U$ being a classical solution to \eqref{PDE X g}. Our contribution is proving that $\delta> 0$ is already enough. With this relaxed H\"{o}lder continuity, \eqref{SDE X} no longer admits a unique strong solution, and the strong formulation in \cite{ET09} cannot be used anymore. For the interior smoothness of $U$, we first rely on a fundamental result in Lieberman \cite{Lib96} (which requires only $\delta>0$) to construct a smooth function solving \eqref{PDE X g} away from the boundary; then, we prove by probabilistic methods that this smooth function coincides with $U$ (Lemma~\ref{lem:smoothness}). For the continuity of $U$ up to the boundary, we take a novel approach by using techniques of viscosity solutions developed in Bayraktar and S\^{i}rbu \cite{BS12}. 
Since \cite{BS12} does not require any regularity of $\sigma$, the continuity of $U$ up to the boundary holds under very general conditions (Lemma~\ref{lem:p-boundary condition}).
This, together with the aforementioned non-standard Feynman-Kac formula, characterizes $U$ as the unique classical solution to \eqref{PDE X g}, under local H\"{o}lder continuity of $\sigma$ with exponent $\delta\in(0,1]$ and \eqref{DS02} (Theorem~\ref{thm:exist and unique}).
This in particular generalizes  \cite[Theorem 2]{BX10}.

A comparison theorem for \eqref{PDE X g} is also proved under \eqref{DS02}, without assuming the usual linear growth condition on $\sigma$. Moreover, we establish the equivalence between ``the uniqueness of solutions to \eqref{PDE X g}'' and ``a comparison theorem for \eqref{PDE X g} holds'' (Theorem~\ref{thm:full equiv}). While the latter clearly implies the former, the converse is nontrivial.

Finally, we turn to (Q3), the case where $U$ is {\it not} guaranteed to be smooth (e.g. due to the lack of local H\"{o}lder continuity of $\sigma$). Assuming that $\sigma$ is positive, continuous and satisfies \eqref{DS02}, we construct $\{\sigma_n\}_{n\in\N}$ of locally H\"{o}lder continuous functions such that $\sigma_n\uparrow\sigma$. Each $\sigma_n$ gives rise to a price process $X^{(n)}$, and the corresponding stochastic solution $U_n$ can be characterized as the unique classical solution to \eqref{PDE X g} with $\sigma$ replaced by $\sigma_n$. By expressing $X^{(n)}_T$ and $X_T$ as time-changed Brownian motions, we show that $X^{(n)}_T\to X_T$ in distribution (Lemma~\ref{lem:X^n to X}), which leads to the characterization of $U$ as a limit of the approximating stochastic solutions $U_n$ (Theorem~\ref{thm:U as a limit} and Remark~\ref{rem:U as a limit}).

The paper is organized as follows. Section~\ref{sec:martingality} presents a new proof of Kotani's criterion for $X$ being a martingale.  Section~\ref{sec:Cauchy} studies the existence and uniqueness of local stochastic solutions. A Feynman-Kac formula is obtained as a by-product. Section~\ref{sec:classical} establishes the interior smoothness of $U$ and the continuity of $U$ up to the boundary. This enables us to characterizes $U$ as the unique classical solution to \eqref{PDE X g}, under conditions more general than those in the literature. A comparison theorem for \eqref{PDE X g} is derived, without the linear growth condition on $\sigma$. Section~\ref{sec:U as a limit} deals with the case where $U$ may not be smooth, and approximates $U$ by a sequence of Cauchy problems.


\section{A New Proof of Kotani's Criterion for Martingality of $X$}\label{sec:martingality}

In this section, we assume without loss of generality that $t=0$ in \eqref{SDE X}, and consider a more general state space $I$:  
\begin{equation}\label{SDE X'}
 d X^{x}_{s} = \sigma(X^{x}_{s}) dW_{s},\quad  X^{x}_{0} = x\in I=(\ell,r),
\end{equation}
where $-\infty\le\ell<r\le\infty$. The endpoints of $I=(\ell,r)$ are assumed to be absorbing for $X$; that is, when $X$ reaches an endpoint in finite time, it stays at the endpoint. 

By the arguments in \cite{ES85, ES91}, or \cite[Theorem 5.5.15]{KS-book-91}, the standing assumption \eqref{standing asm'} implies that \eqref{SDE X} admits a weak solution, unique in the sense of probability distribution, up to the ``exit time''
\[
\tau_\infty := \lim_{n\to\infty} \tau_n,\quad \tau_n:= \inf\{t\ge 0: X_t\notin (\ell+1/n,r-1/n)\}.
\]
The process $X$ is by definition a local martingale, and it is of interest to establish a criterion for $X$ being a true martingale. If $\ell$ and $r$ are both finite,  as a bounded local martingale, $X$ is trivially a martingale. When either $\ell$ or $r$ is infinite, the following criterion holds.

\begin{theorem}\label{thm:X true mart}  
Let $(X, W, \Omega,\F,\P, \{\F_{s}\}_{s\ge 0})$  be a weak solution to \eqref{SDE X'}. 
\begin{itemize}
\item [(i)] If $\ell\in \R$ and $r=\infty$, then $X$ is a martingale if and only if 
\begin{equation}\label{DS02 l-finite}
\int_c^\infty \frac{x}{\sigma^2(x)} dx =\infty\quad \hbox{for some $c\in I$.}
\end{equation}
\item [(ii)] If $\ell=-\infty$ and $r\in\R$, then $X$ is a martingale if and only if 
\begin{equation}\label{DS02 r-finite}
\int_{-\infty}^c \frac{-x}{\sigma^2(x)} dx =\infty\quad \hbox{for some $c\in I$.}
\end{equation}
\item [(iii)] If $\ell=-\infty$ and $r=\infty$, then $X$ is a martingale if and only if both \eqref{DS02 l-finite} and \eqref{DS02 r-finite} hold.
\end{itemize}
\end{theorem}

\begin{remark}\label{rem:literature}
For the special case $I=(0,\infty)$, Theorem~\ref{thm:X true mart} (i) reduces to: $X$ is a martingale if and only if \eqref{DS02} holds. This result was first established in Delbaen and Shirakawa \cite[Theorem 1.6]{DS02}, under the stronger assumption that $\sigma$ and $1/\sigma$ are positive and locally bounded. Through studying exponential martingales, Mijatovi\'{c} and Urusov recovered this result in \cite[Corollary 4.3]{MU12}, without the stronger assumption on $\sigma$. Karatzas \& Ruf \cite{Karatzas-Ruf-13} recently derived the same result by analyzing the $h$-transform of $X$; see \cite[Corollary 6.5]{Karatzas-Ruf-13}.

For the general case $I=(\ell,r)$ with $-\infty\le \ell<r\le\infty$, Theorem~\ref{thm:X true mart} was first established in Kotani \cite[Theorem 1]{Kotani06} in the form of a short lecture note. Hulley and Platen \cite[Section 3]{HP08} gives a comprehensive exposition, with detailed proofs, of Kotani's result. Kotani's analytic method utilizes  the fundamental solutions to the associated generalized differential equation. 
We provide a probabilistic proof, which is very different from and more concise than the arguments in \cite{Kotani06,HP08}. 
\end{remark}

\begin{remark}
   In view of Feller's classification of boundary points (see, for example, Chapter 15 of Karlin and Taylor \cite{KarlinT}), since $X$ of  \eqref{SDE X}  is on natural scale, \eqref{DS02 l-finite} implies that the right boundary $\infty$ is a natural boundary. Likewise, \eqref{DS02 r-finite} implies that the left boundary $-\infty$ is a natural boundary. Therefore we can also restate    Theorem \ref{thm:X true mart} as follows: $X$ is a martingale if and only if all its infinite boundary points   are natural boundaries. 
\end{remark}

We first present a useful lemma. 

\begin{lemma}\label{lem:mtgl}
Let $Y:\Omega\times\R_+\mapsto \R$ be a continuous local martingale defined on some filtered probability space $(\Omega,\F,\{\F_t\}_{t\ge 0},\P)$, with $Y_0 = y\in\R$ $\P$-a.s. If $Y$ is bounded from below $\P$-a.s., then $Y$ is a $\P$-martingale if and only if 
\begin{equation}\label{nP[tau_n<T]}
\lim_{\beta\to \infty} \beta \mathbb{P}[\tau^\beta < T]  = 0\quad   \hbox{for all $T\ge 0$},
\end{equation}
where $\tau^\beta:= \inf\{t\ge0: Y_t \ge \beta\}$.
\end{lemma}

\begin{proof}
Since $Y_{\cdot\wedge \tau^\beta}$ is a bounded local martingale and thus a martingale, for all $T\ge 0$, 
\begin{equation}\label{y=E[Y]}
y = \mathbb E[Y_{\tau^\beta\wedge T}] = \mathbb E [Y_{\tau^\beta} 1_{\{\tau^\beta<T\}}] + \mathbb E[Y_{T} 1_{\{\tau^\beta\ge T\}}]= \beta \mathbb{P}[\tau^\beta < T]  + \E[Y_T 1_{\{\tau^\beta\ge T\}}].
\end{equation}
As a local martingale bounded from below, $Y$ is a supermartingale. The supermartingale convergence theorem, see e.g. \cite[Theorem 1.3.15]{KS-book-91}, asserts that $Y_\infty(\omega):=\lim_{t\to\infty}Y_t(\omega)$ exists for a.e. $\omega\in\Omega$ and $\E[|Y_\infty|]<\infty$. This implies that $Y$ does not explode in finite time $\P$-a.s., and thus $\tau^{\beta} \uparrow \infty\ \P\hbox{-a.s.}$ as $\beta \uparrow \infty$. Then, applying the monotone convergence theorem to \eqref{y=E[Y]} yields $\mathbb E[ Y_{T}] = y$ for all $T\ge 0$ if and only if \eqref{nP[tau_n<T]} holds. Finally, since $Y$ is a supermartingale and thus $\E[Y_t\mid\F_s]\le Y_s$ $\P$-a.s. for all $0\le s\le t$, $Y$ is a martingale if and only if $\mathbb E[ Y_{T}] = y$ for all $T\ge 0$.
\end{proof}

The result of this lemma can be traced back to equation (5) in Cox and Hobson \cite{CH05}, where the authors did not prove the result, but cited the lecture note by Az{\'e}ma, Gundy, and Yor \cite{AGY80}. Here, we provide a simpler proof of the same result. Note that results similar to Lemma~\ref{lem:mtgl}, but for much more general processes $Y$, have been established in Elworthy, Li, and Yor \cite{ELY99}; see e.g. Lemma 2.1 and Lemma 3.2 in \cite{ELY99}.

\begin{proof}[Proof of Theorem~\ref{thm:X true mart}]
(i) Without loss of generality, we assume that $\ell=0$, and then \eqref{DS02 l-finite} reduces to \eqref{DS02}. 
Indeed, if $\ell\neq 0$, one may consider the shifted process $\widetilde X:= X- \ell$, which admits the dynamics $d\widetilde X_t = \widetilde\sigma(\widetilde X) dW_t$, with $\widetilde\sigma(z) := \sigma(z+\ell)$ satisfying \eqref{standing asm'} on $(0,\infty)$. Direct calculation shows that $\widetilde\sigma$ satisfying \eqref{DS02} is equivalent to $\sigma$ satisfying \eqref{DS02 l-finite}.


Define the function $\hat \sigma(y): = \frac{\sigma(y)}{1+ y}$. Consider the stochastic differential equation (SDE)
\begin{equation}
\label{SDE Y}
dY_{t} = \sigma(Y_{t})[\hat \sigma (Y_{t})dt+d W_{t} ], \quad Y_{0} = x \in I=(0,\infty), 
\end{equation} 
and assume that $Y$ is absorbed at $0$ once it arrives there. Thanks to \eqref{standing asm'}, we may conclude from Theorem 5.5.15 of \cite{KS-book-91} that the SDE \eqref{SDE Y} admits a weak solution $(Y, W', \Omega', \F', \P', \{\F_{s}'\}_{s\ge 0})$, unique in distribution, up to a possibly finite explosion time $S_\infty$, defined by
\[
S_\infty(\omega):=\lim_{n\to\infty}S_n(\omega),\quad \hbox{with}\quad S_{n}(\omega): = \inf\{t \ge 0: Y_{t}(\omega) \ge n \}\quad \forall n\in\N.
\]
By \eqref{standing asm'} again, conditions (ND)$'$ and (LI)$'$ on p. 343 of \cite{KS-book-91} are satisfied. We can therefore apply to the process $Y$ Feller's test for explosions (see, for example, Theorem 5.5.29 in \cite{KS-book-91}), which implies that $\P'\{S_{\infty} < \infty \} =0$ if and only if  $v(\infty-)=\infty$, with $v(\cdot)$ defined as
\begin{equation}\label{v_c}
v(z):= 2 \int_{1}^{z} \Big (\frac 1 {1+y} - \frac 1 {1+z} \Big) \frac{(1 + y)^{2}}{\sigma^{2}(y)} d y.
\end{equation}
By direct computation, 
\begin{equation}\label{v_c calculation}
 \frac{z}{1+z} \int_{1}^{z/2} \frac{y}{\sigma^{2}(y)} dy \le v_{}(z) \le 4 \int_{1}^{z} \frac{y}{\sigma^{2}(y)} dy.
\end{equation}
We therefore conclude that $\P'\{S_{\infty} < \infty \} =0$ if and only if \eqref{DS02} holds.


Now, consider the process $Z_{t} : = \frac{1+x}{1+Y_{t}}$ for $t \ge 0$. By definition, $Z$ is a bounded process taking values in $ [0, 1+x]$. Moreover, it satisfies the dynamics $d Z_{t} = - Z_{t} \hat \sigma(Y_{t}) d W'_{t}$,  thanks to It\^o's formula. Thus, $Z$ is a bounded $\P'$-local martingale, whence a $\P'$-martingale, with respect to $\{\F_{t}'\}_{t\ge 0}$. For each $n\in\N$, consider the stopping time $\widetilde S_n := S_n\wedge n$, and note that
\begin{equation}\label{lim=lim}
\lim_{n\to\infty}\widetilde S_n(\omega) = \lim_{n\to\infty}S_n (\omega)= S_\infty(\omega),\quad \forall \omega\in \Omega'. 
\end{equation}
For each $n\in\N$, define the probability measure $\Q_{n}$ on $(\Omega',\F'_{\widetilde S_{n}})$ by $d \Q_{n} : = Z_{\widetilde S_{n} } d \P'$. 
Using the martingality of $Z$, we can verify that $\Q_{n}$, $n\in \mathbb N$, are consistent in the sense that $\Q_{n+1} = \Q_{n}$ on $\F'_{\widetilde S_{n}}$. In addition,   for any $t \ge 0$, 
 \begin{displaymath}
\lim_{n\to \infty} \Q_{n}\{ \widetilde S_{n}\le t \} = \lim_{n\to \infty}  \E^{\P'} [ Z_{\widetilde{S}_{n}} 1_{\{ \widetilde{S}_{n} \le t \}}] =  \lim_{n\to \infty}  \E^{\P'} [ Z_{{S}_{n}} 1_{\{ {S}_{n} \le t \}}] =\lim_{n\to\infty}  \frac{1+x}{1+n}\P'(S_{n} \le t) =0,
\end{displaymath} 
where the second equality follows from the observation that for all $n>t$, we have (a) $\widetilde S_n\le t$ if and only if $S_n\le t$, and (b) $S_n\le t$ implies $Z_{\widetilde S_n} =Z_{S_n}$. Thus, by Theorem 1.3.5 of \cite{Stroock-Varadhan-book-06}, there exists a unique probability $\Q$ on $(\Omega', \F') $ such that $\Q =\Q_{n}$ on $\F_{\widetilde S_{n}}'$ for all $n \in \mathbb N $. 

Observe that the process 
$\widetilde W_{t} : = W'_{t}  + \int_{0}^{t} \hat\sigma(Y_{s}) ds$ 
is well-defined up to $S_n$, for all $n\in\N$. Indeed,  
since $Y$ is a weak solution to \eqref{SDE Y} up to $S_\infty$, $\P'\{\int_{0}^{t\wedge S_{n}}\sigma^{2} (Y_{s})ds < \infty\} =1$ for all $t\ge 0$ and $n\in \mathbb N$ (Definition 5.5.20 of \cite{KS-book-91}). It follows that for any fixed $t\ge 0$, it holds $\P'$-a.s. that for all $n\in \mathbb N$,
\begin{displaymath}
\int_{0}^{t\wedge S_{n}}|\hat\sigma  (Y_{s})|ds = \int_{0}^{t\wedge S_{n}} \frac{|\sigma(Y_{s})|}{1+ Y_{s}} ds \le  \int_{0}^{t\wedge S_{n}}  | \sigma(Y_{s})| ds \le t^{1/2} \biggl( \int_{0}^{t\wedge S_{n}}\sigma^{2} (Y_{s})ds \biggr)^{1/2} < \infty.
\end{displaymath} 
Now, by Girsanov's theorem, 
$\widetilde W_{\cdot\wedge \widetilde S_{n}}$ is a $\Q_{n}$- (and hence a $\Q$-) Brownian motion on $[0,\widetilde{S}_{n}]$, for all $n \in \mathbb N$. Thanks to \eqref{lim=lim} and Lemma \ref{l:01}, $\widetilde W_{\cdot\wedge S_{\infty}-} := \lim_{s\uparrow(\cdot\wedge S_\infty)} \widetilde W_s$ is a $\Q$-Brownian motion on $[0, S_{\infty})$. Consequently, for any $\Q$-Brownian motion $\{B_t\}_{t\ge 0}$, the process $\widehat W_{t}: = \widetilde W_{t\wedge S_{\infty}-} + B_{t\vee S_{\infty}} - B_{S_{\infty}}$ is a $\Q$-Brownian motion on $[0,\infty)$. In view of the $\P'$-dynamics \eqref{SDE Y}, $Y$ satisfies under $\Q$ the dynamics   
\begin{displaymath}
 Y_{t\wedge S_{n}} = x + \int_{0}^{t\wedge S_{n}} \sigma(Y_{s}) d \widehat W_{s} \quad t \ge 0,\quad \forall n \in \mathbb N. 
\end{displaymath} 
In other words, $(Y,\widehat W, \Omega', \F', \Q, \{\F_{t}'\}_{t\ge 0})$ is a weak solution to  \eqref{SDE X'} up to $S_\infty$. But a weak solution to  \eqref{SDE X'} does not explode in finite time (Problem 5.5.3 of \cite{KS-book-91}), and thus $S_\infty = \infty$ $\Q$-a.s. 

Consider $T_{n}: = \inf\{ t\ge 0: X_{t} \ge n\}$, for all $n\in \mathbb N$. Since weak solutions to \eqref{SDE X'} are unique in the sense of probability distribution, for each fixed $T\ge 0$ we must have $\P(T_{n} \le  T)=\Q(S_{n}  \le  T)$ for all $n\in\N$. It follows that for each $T\ge 0$,
 \begin{equation}
\label{eq5}
\begin{aligned}
  \lim_{n\to\infty} n \P(T_{n} <  T)& =  \lim_{n\to\infty} n \Q(S_{n}  <  T) =  \lim_{n\to\infty} n \Q_{n}(S_{n}  <  T)\\
  & =   \lim_{n\to\infty} n\E^{\P'}[Z_{S_{n}} 1_{\{S_{n}  <  T\}}]  =  \lim_{n\to\infty} \frac{n(1+ x)}{1+n}\P'(S_{n}  <  T)\\ 
  & = (1+ x)  \lim_{n\to\infty}  \P'(S_{n}  <  T) . 
\end{aligned} \end{equation}
By Lemma \ref{lem:mtgl}, \eqref{eq5} gives the equivalence between ``$X$ is a martingale'' and ``$\lim_{n\to\infty}  \P'(S_{n}  <  T)=0$ for all $T\ge 0$''. On the other hand, since $S_\infty < \infty$ if and only if there exists $T < \infty$ such that $S_n < T$ for all $n\in\N$,
\[
\P'\{S_{\infty} < \infty \} = \P' \Bigg\{ \bigcup_{T =1}^{\infty} \bigcap_{n=1}^{\infty} \{S_{n} < T \} \Bigg\} =\lim_{T\to \infty} \lim_{n\to \infty} \P'\{S_{n} < T\}.
\]
As the map $T \mapsto \lim_{n\to\infty}  \P'(S_{n}  \le  T) $ is nondecreasing, the above equation implies ``$\lim_{n\to\infty}  \P'(S_{n}  <  T)=0$ for all $T\ge 0$'' if and only if ``$\P'\{S_{\infty} < \infty \}=0$''. Thus, we obtain the equivalence between ``$X$ is a martingale'' and ``$\P'\{S_{\infty} < \infty \}=0$''. The proof is complete once we recall that the latter condition is equivalent to \eqref{DS02}. 

(ii) The process $Y_\cdot := - X_\cdot$ is a local martingale evolving in $\widetilde I:=(-r,\infty)$, satisfying the dynamics $dY_t = \widetilde\sigma(Y_t)dW_t$, where $\widetilde\sigma(z) := -\sigma(-z)$. Applying part (i) to $Y$, we have $Y$ is a martingale if and only if $\int_{c}^\infty \frac{z}{\widetilde{\sigma}^2(z)}dz = \infty$ for some $c\in\widetilde I$. By the definition of $\widetilde\sigma$, this is equivalent to \eqref{DS02 r-finite}.

(iii) Consider $I=(-\infty,\infty)$. We will first show that $\E[|X_T|]< \infty$ for all $T\ge 0$, by using an argument in \cite[Theorem 3.10 (a)]{HP08}. Let $\{L^a_t:a\in I, t\ge 0\}$ be the family of local times of $X$, where $L^a_t = \lim_{\eps\downarrow 0 }\frac{1}{2\eps}\int_0^t 1_{\{X_s \in (a-\eps,a+\eps)\}} \sigma^2(X_s) ds$ $\P$-a.s. Note that \cite[Corollary 1.8, p.226]{RY99} particularly implies that we can choose a version of $L_{t}^{a}$,  such that $a\mapsto L_{t}^{a}$ is continuous for each $t\ge 0$ $\P$-a.s. For any $T\ge 0$, Takana's formula (see e.g. \cite[p.222]{RY99})  and Fatou's Lemma lead to $\E[|X_T|] \le |x| + \E[L^0_T]$.
Fix $\delta>0$. By the occupation times formula (see e.g. \cite[Corollary 1.6, p.224]{RY99}) and the mean value theorem, for $\P$-a.e. $\omega\in\Omega$, there exists $y^*\in (-\delta,\delta)$ such that
\begin{equation}\label{occupation time}
\int_0^T 1_{ (-\delta,\delta)}(X_{t}(\omega)) dt = \int_I 1_{(-\delta,\delta)}(y) L^y_T(\omega) dy = \int_{-\delta}^{\delta} L^y_T(\omega) dy = 2\delta L^{y^*}_T(\omega).
\end{equation}
For any $\alpha\in (0,1/2)$, \cite[Corollary 1.8, p.226]{RY99} asserts that there exists a 
$K = K(T)> 0$ such that 
$\sup\{|L^a_t-L^0_t|:t\in[0,T]\} \le K(T) |a|^\alpha$ $\P$-a.s. for all $a\in\R$. This, together with \eqref{occupation time}, gives
\[
\E[L^0_T] \le \E[K(t)\delta^\alpha + L^{y^*}_T] = K(T)\delta^\alpha + \frac{1}{2\delta}\int_0^T \P[X_t \in (-\delta,\delta)] dt\le K(T)\delta^\alpha + \frac{T}{2\delta},
\]
which implies $\E[|X_T|] \le |x|+ K(T)\delta^\alpha + \frac{T}{2\delta}<\infty$. 

Now, if $X$ is a martingale on $(-\infty,\infty)$, then \eqref{DS02 l-finite} and \eqref{DS02 r-finite} hold as a consequence of parts (i) and (ii). To prove the converse, suppose \eqref{DS02 l-finite} and \eqref{DS02 r-finite} hold, and let $S^+_n:=\inf\{s\ge 0: X_s\ge n\}$ and $S^-_{n}:=\inf\{s\ge 0: X_s\le -n\}$. By parts (i) and (ii), $X_{\cdot\wedge S^-_{n}}$ and $X_{\cdot\wedge S^+_n}$ are martingales, for all $n\in\N$. Thus, for any $T\ge 0$,  
\begin{align*}
x &= \E[X_{S^+_n\wedge T}] = \E [X_{S^+_n} 1_{\{S^+_n<T\}}] + \E[X_{T} 1_{\{S^+_n\ge T\}}]= n \P\left[S^+_n < T\right]   + \E [X_T 1_{\{S^+_n\ge T\}}],\\
x &= \E[X_{S^-_n\wedge T}] = \E [X_{S^-_n} 1_{\{S^-_n<T\}}] + \E[X_{T} 1_{\{S^-_n\ge T\}}]= -n \P\left[S^-_n < T\right]   + \E [X_T 1_{\{S^-_n\ge T\}}].
\end{align*}
Since $X$ does not explode in finite time $\P$-a.s. (see e.g. \cite[Problem 5.5.3]{KS-book-91}), $S^+_n\uparrow\infty$ and $S^-_n\uparrow\infty$ $\P$-a.s. as $n\uparrow\infty$. Thanks to $\E[|X_T|]<\infty$, applying the dominated convergence theorem to the above equations yield  
\begin{equation*}
x = \lim_{n\to\infty} n \P\left[S^+_n < T\right] + \E[X_T]\quad \hbox{and} \quad x = -\lim_{n\to\infty} n \P\left[S^-_n < T\right] + \E[X_T],
\end{equation*}
which implies $x \ge \E[X_T]$ and $x\le \E[X_T]$. Thus, $x=\E[X_T]$ for all $T\ge 0$, and we conclude that $X$ is a martingale.
\end{proof}


\section{Local Stochastic Solutions} \label{sec:Cauchy}

\subsection{Preliminaries}\label{sec:local stoch sol}

In this section, we introduce the notion of local stochastic solutions to a linear parabolic equation, and present several existence and uniqueness results.

Let $G$ be an open subset of $\R^d$. For any $(t,v)\in (0,\infty)\times G$, we consider an $\R^d$-valued diffusion 
\begin{equation}\label{SDE V}
dV_s=b(V_s)ds+a(V_s)dW_s,\ V_t=v\in G,
\end{equation}
where $b:\R^d\mapsto\R^d$ and $a:\R^d\mapsto \mathbb{M}^d$ are Borel functions (here $\mathbb{M}^d$ denotes the set of $d\times d$ matrices), and $W$ is a standard $d$-dimensional Brownian motion. We only assume that $b$ and $a$ are such that \eqref{SDE V} has a weak solution. Define the differential operator
\[
L^{a,b}:=b'\nabla+\frac{1}{2}\text{Tr}(aa'\nabla^2).
\]
Fix a time horizon $T>0$. Given a Borel function $g:\bar{G}\mapsto\R$, we consider the following boundary value problem
\begin{equation}\label{PDE}
\begin{cases}
\partial_t u(t,v)+L^{a,b}u(t,v)=0,\ \ &(t,v)\in(0,T)\times G;\\
u(t,v)=g(v),\ \ &(t,v)\in [0,T]\times\partial{G}\cup\{T\}\times G.
\end{cases}
\end{equation}

\begin{definition}
We say $u:[0,T]\times\bar{G}\mapsto\R$ is a classical solution to \eqref{PDE} if $u\in C^{1,2}([0,T)\times G) \cap C([0,T]\times \bar G)$ and satisfies \eqref{PDE}.
\end{definition}

Under current general set-up, \eqref{PDE} may not admit a classical solution. Motivated by \cite[Section 5]{SV72} and the specific formulation in \cite[Definitions 2.1 and 2.2]{BS12}, we introduce a weaker notion of solutions to \eqref{PDE} as follows.

\begin{definition}\label{defn:stoch sol V}
A Borel function $u:[0,T]\times \bar{G}\mapsto\R$ is a stochastic solution (resp. local stochastic solution) to \eqref{PDE} if
\begin{itemize}
\item [(i)] for any $(t,v)\in[0,T]\times G$ and any weak solution $(V^{t,v}, W^{t,v},\Omega^{t,v},\F^{t,v},\P^{t,v},\{\F^{t,v}_s\}_{s\ge t})$ of \eqref{SDE V},  $u\left(r\wedge\tau^{t,v},V^{t,v}_{r\wedge\tau^{t,v}}\right)$ is a $\P^{t,v}$-martingale (resp. $\P^{t,v}$-local martingale), where
\[
\tau^{t,v}:=\inf\{s\ge t: V^{t,v}_s\notin G\}\wedge T.
\]
\item [(ii)] $u(t,v)=g(v)$ for $(t,v)\in[0,T]\times\partial{G}\cup\{T\}\times G$.
\end{itemize}
\end{definition}

\begin{remark}\label{p:class-stoch}
In \cite{SV72} and \cite{BS12}, the authors introduce stochastic solutions, but not their {\it local} versions. Since they only take bounded functions as solutions, every local stochastic solution is a stochastic solution under their consideration. Here, without the boundedness restriction, the additional notion of local stochastic solutions is necessary. 
In particular, by applying It\^{o}'s rule to the process $u(r,X^{t,x}_r)$, 
 a classical solution $u$ to \eqref{PDE} is a local stochastic solution to \eqref{PDE}. However, a classical solution $u$ to \eqref{PDE}  is not necessarily a stochastic solution. Example~\ref{eg:x loc sol} below serves as an counterexample.

\end{remark}

\begin{remark}
Our formulation in Definition~\ref{defn:stoch sol V} is more general than those in the literature (see e.g. \cite[Definition 3.1]{Hsu85}, \cite[Defnition 2.2]{JT06}, and \cite[Definition 5.1]{BKX12}), as it requires neither regularity nor boundedness of solutions and allows for local martingale property.  
\end{remark}

Uniqueness of stochastic solutions is inherent in Definition~\ref{defn:stoch sol V}.

\begin{proposition}\label{prop:unique stoch sol}
If \eqref{PDE} admits a stochastic solution $u$, then $u(t,v)=\E^{t,v}[g(V^{t,v}_{\tau^{t,v}})]$ for any weak solution $(V^{t,v}, W^{t,v},\Omega^{t,v}, \F^{t,v},\P^{t,v},\{\F^{t,v}_s\}_{s\ge t})$ to \eqref{SDE V}.
\end{proposition}

\begin{proof}
If $u$ is a stochastic solution to \eqref{PDE}, then
\[
u(t,v)=\E^{t,v}[u(T\wedge\tau^{t,v},V^{t,v}_{T\wedge\tau^{t,v}})]=\E^{t,v}[u(\tau^{t,v},V^{t,v}_{\tau^{t,v}})]=\E^{t,v}[g(V^{t,v}_{\tau^{t,v}})].
\]
\end{proof}

In spite of Proposition~\ref{prop:unique stoch sol}, the function $u(t,v):=\E^{t,v}[g(V^{t,v}_{\tau^{t,v}})]$ a priori may not be a stochastic solution to \eqref{PDE}. Indeed, since we currently assume only existence, but not uniqueness in distribution, of weak solutions to \eqref{SDE V}, the function $u$ may not be well-defined. For instance, when $G$ is a subset of $\R$ and $b\equiv 0$, ``\eqref{SDE V} admits a weak solution while uniqueness in distribution fails'' can be completely characterized by the condition
\[
\left\{x\in\R:\int_{-\eps}^\eps\frac{dy}{a^2(x+y)}=\infty,\ \forall\eps>0\right\}\subsetneqq \{x\in\R:a(x)=0\}.
\]
We refer the readers to Section 5.5 A and B in \cite{KS-book-91} for details. In particular, Remark 5.5.6 explains how uniqueness fails and provides a few concrete examples.


In the following, we provide a sufficient condition for $u(t,v):=\E^{t,v}[g(V^{t,v}_{\tau^{t,v}})]$ being a stochastic solution to \eqref{PDE}.

\begin{proposition}\label{prop:E[] is sol}
Assume uniqueness in distribution of weak solutions to \eqref{SDE V}, for any initial condition $(t,v)\in[0,\infty)\times \bar{G}$. If $g(V^{t,v}_{\tau^{t,v}})$ is $\P^{t,v}$-integrable  for all  $(t,v)\in[0,T]\times \bar{G}$, then the function $u(t,v):=\E^{t,v}[g(V^{t,v}_{\tau^{t,v}})]$
is the unique stochastic solution to \eqref{PDE}.
\end{proposition}

\begin{proof}
Thanks to the uniqueness in distribution of weak solutions to \eqref{SDE V}, the function $u(t,v):=\E^{t,v}[g(V^{t,v}_{\tau^{t,v}})]$ is well-defined, and the unique weak solution $V^{t,v}$ to \eqref{SDE V} has the strong Markov property (see e.g. \cite[Corollary 4.23]{ES91}). Fix $t\ge 0$. For simplicity, we write $\ell^{t,v}$ for $(\ell\wedge\tau^{t,v},V^{t,v}_{\ell\wedge\tau^{t,v}})$ for any $\ell\ge t$. For any $t\le s\le r$, we have
\begin{align*}
\E^{t,v}\left[u\left(r\wedge\tau^{t,v},V^{t,v}_{r\wedge\tau^{t,v}}\right)\ \middle|\ \F_s\right]
&=\E^{t,v}\left[\E^{r^{t,v}}\left[g \left(V^{r^{t,v}}_{\tau^{r^{t,v}}}\right)\right]\ \middle|\ \F_s\right]=\E^{t,v}\left[\E^{{t,v}}\left[g\left(V^{{t,v}}_{\tau^{{t,v}}}\right)\ \middle|\ \F_r\right]\ \middle|\ \F_s\right]\\
&= \E^{{t,v}}\left[g\left(V^{{t,v}}_{\tau^{{t,v}}}\right)\ \middle|\ \F_s\right]=
\E^{s^{t,v}}\left[g\left(V^{s^{t,v}}_{\tau^{s^{t,v}}}\right)\right]=u\left(s\wedge\tau^{t,v},V^{t,v}_{s\wedge\tau^{t,v}}\right),
\end{align*}
where the Markov property of $V^{t,v}$ is used in the second and the fourth equalities. This shows that $u\left(r\wedge\tau^{t,v},V^{t,v}_{r\wedge\tau^{t,v}}\right)$ is a martingale. Since $u(t,v)=g(v)$ for $(t,v)\in[0,T]\times\partial{G}\cup\{T\}\times G$ by definition, we conclude that $u$ is a stochastic solution to \eqref{PDE}. The uniqueness follows from Proposition~\ref{prop:unique stoch sol}.
\end{proof}

\begin{remark}\label{rem:relations between sols}
In view of \cite[Theorem 2.7]{JT06}, one may expect the relation``a {\it continuous} local stochastic solution to \eqref{PDE} must be a classical solution''. However, since \cite[Theorem 2.7]{JT06} relies on the arguments in \cite[Theorem 16.1]{Lieberman92}, which demands H\"{o}lder continuity of the coefficients $a$ and $b$ in \eqref{SDE V}, it is unclear if the expected relation holds under current general set-up. It is of interest to either establish the relation using different arguments or find a counterexample.  

In Section~\ref{sec:classical}, we will impose H\"{o}lder continuity on the coefficient $\sigma$ of \eqref{SDE X}. We can then prove the above relation by the same arguments in \cite[Theorem 2.7]{JT06}. This relation, nonetheless, is not useful in Section~\ref{sec:classical}, as proving the continuity of the stochastic solution is difficult; see Section~\ref{sec:interior}.
\end{remark}


\subsection{Local Stochastic Solutions to the Cauchy Problem \eqref{PDE X g}}

In this section, we apply the set-up in Section~\ref{sec:local stoch sol} to the Cauchy problem \eqref{PDE X g}. Our main result is a characterization for the existence of a unique local stochastic solution to \eqref{PDE X g} (Theorem~\ref{thm:unique local stoch sol}). This in particular gives a non-standard Feynman-Kac formula (Corollary~\ref{coro:FK}).

Consider the process $X$ in \eqref{SDE X} with $I=(0,\infty)$. The corresponding boundary value problem \eqref{PDE} becomes the Cauchy problem \eqref{PDE X g}. From the fact that $X$ is non-exploding (see e.g. \cite[Problem 5.5.3]{KS-book-91}) and is absorbed at the origin once it arrives there, we observe that
\begin{equation*}
u\left(r\wedge\tau^{t,x},X^{t,x}_{r\wedge\tau^{t,x}}\right)=u(r\wedge T,X^{t,x}_{r\wedge T})\ \ \hbox{for all}\ r\ge t,
\end{equation*}
where $\tau^{t,x}=\inf\{s\ge t: X^{t,x}_s\notin (0,\infty)\}\wedge T$. Definition~\ref{defn:stoch sol V} then reduces to:


\begin{definition}\label{defn:stoch sol X}
A Borel function $u:[0,T]\times [0,\infty) \mapsto\R$ is a stochastic solution (resp. local stochastic solution) to \eqref{PDE X g} if
\begin{itemize}
\item [(i)] for any $(t,x)\in[0,T]\times [0,\infty)$ and any weak solution $(X^{t,x}, W^{t,x},\Omega^{t,x},\F^{t,x},\P^{t,x},\{\F^{t,x}_s\}_{s\ge t})$ of \eqref{SDE X},  $u\left(r\wedge T,X^{t,x}_{r\wedge T}\right)$ is a $\P^{t,x}$-martingale (resp. $\P^{t,x}$-local martingale).
\item [(ii)] $u(T,x)=g(x)$ for $x\in (0, \infty)$, $u(t,0)=g(0)$ for $t\in[0,T]$.
\end{itemize}
\end{definition}

The next example shows that a classical solution may not be a stochastic solution.

\begin{example}\label{eg:x loc sol}
With $g(x):=x$, the Cauchy problem \eqref{PDE X g} is
\begin{equation}\label{PDE X}
\begin{cases}
\partial_t u+\frac{1}{2}\sigma^2(x)\partial_{xx}u=0,\ \ &(t,x)\in[0,T)\times (0,\infty);\\
u(T,x)=x,\ \ &x\in(0,\infty);\\
u(t,0)=0,\ \ &t\in[0,T].
\end{cases}
\end{equation}
While $u(t,x):=x$ is a classical solution to \eqref{PDE X}, it is in general only a local stochastic solution. It is a stochastic solution if and only if $X$ is a true martingale.
\end{example}

Let $\hat D$ denote the set of functions which have linear growth in $x$, i.e.
$$\hat D :=\{u:[0,T]\times [0,\infty)\mapsto\R:\exists K>0\ \hbox{s.t.}\ |u(t,x)|\le K(1+x)\ \hbox{for all}\ (t,x)\in[0,T]\times [0,\infty)\}.$$
For any $v:[0,\infty)\mapsto\R$, we will slightly abuse the notation by saying that $v\in\hat{D}$ if there exists $K>0$ such that $|v(x)|\le K(1+x)$ for all $x\in[0,\infty)$.

\begin{proposition}\label{prop:E[T] is sol}
If $g:[0,\infty)\mapsto\R$ belongs to $\hat D$, then $U$ defined in \eqref{defn U} is the unique stochastic solution to \eqref{PDE X g}. Moreover, $U$ belongs to $\hat D$.
\end{proposition}

\begin{proof}
First, recall that the standing assumption \eqref{standing asm'} implies the existence of a weak solution to \eqref{SDE X} unique in distribution (as explained in the beginning of Section~\ref{sec:martingality}). Moreover, take $K>0$ such that $|g(x)|\le K(1+x)$ for all $x\in[0,\infty)$. Then, for any $[0,T]\times [0,\infty)$, $|\E^{t,x}[g(X^{t,x}_{\tau^{t,x}})]|\le \E^{t,x}[|g(X^{t,x}_{\tau^{t,x}})|]\le \E^{t,x}[K(1+X^{t,x}_{\tau^{t,x}})]\le K(1+x)<\infty$.  
We can now conclude from Proposition~\ref{prop:E[] is sol} that the function $u(t,x):=\E^{t,x}[g(X^{t,x}_{\tau^{t,x}})]$ is the unique stochastic solution to \eqref{PDE X g}, and $u\in\hat D$. Finally, we observe that $u$ coincides with $U$ defined in \eqref{defn U}:
\[
u(t,x)=\E^{t,x}[g(X^{t,x}_{\tau^{t,x}})]=\E^{t,x}[g(0)1_{\{\tau^{t,x}<T\}}+g(X^{t,x}_T)1_{\{\tau^{t,x}=T\}}]=\E^{t,x}[g(X^{t,x}_T)]=U(t,x),
\]
where the third equality follows from $X^{t,x}$ being absorbed at the origin once it arrives there.
\end{proof}

The next result connects the uniqueness of local stochastic solutions to the martingality of $X$. Parts of the proof are motivated by \cite{BX10}.

\begin{proposition}\label{prop:equiv}
The following are equivalent:
\begin{itemize}
\item [(i)] $X$ is a martingale.
\item [(ii)] For any $T>0$, \eqref{PDE X} admits a unique local stochastic solution in $\hat D$.
\item [(iii)] For any $T>0$ and $g\in\hat D$, \eqref{PDE X g} admits a unique local stochastic solution in $\hat D$.
\end{itemize}
\end{proposition}

\begin{proof}
$(i)\Leftarrow(ii)$: Suppose $X$ is a strict local martingale. Then there must exist a $T^*> 0$ such that $U(t,x) := \mathbb E^{t,x}[X^{t,x}_{T^*}]<x$. In view of Proposition~\ref{prop:E[T] is sol} and Example~\ref{eg:x loc sol}, $U(t,x)$ and $x$ are two distinct local stochastic solutions to \eqref{PDE X} in $\hat D$ with the time horizon $T:=T^*$.

$(i)\Rightarrow(ii)$: Suppose $X$ is a martingale. Fix an arbitrary $T>0$, and let $u$ be a local stochastic solution to \eqref{PDE X} in $\hat D$. Since $Z_r:=u(r\wedge T, X^{t,x}_{r\wedge T})$ is a local martingale and there exists some $K>0$ such that $|u(t,x)|\le K(1+x)$ for all $(t,x)\in[0,T]\times[0,\infty)$, $Z_{r\wedge\tau^\beta}$ is a martingale for each $\beta>0$, where $\tau^{\beta} := \inf\{s\ge t: X^{t,x}_s \ge \beta\}$. It follows that
\begin{equation} \label{eq:sqs1}
 u(t,x) =  \E^{t,x}[ Z_{T\wedge \tau^{\beta}}] = \E^{t,x}[u({\tau^\beta},\beta) 1_{\{\tau^\beta<T\}}]+\E^{t,x}[X^{t,x}_T 1_{\{\tau^\beta \ge T\}}].
\end{equation}
Note that
\[
\lim_{\beta\to \infty} \left|\E^{t,x}[u({\tau^\beta},\beta) 1_{\{\tau^\beta<T\}}]\right| \le \lim_{\beta\to \infty} K(1+\beta)\P^{t,x}[\tau^\beta<T] =0,
\]
where the equality follows from Lemma~\ref{lem:mtgl}. Also, since $X^{t,x}$ is non-explosive (as a martingale bounded from below), $\tau^{\beta} \uparrow \infty$ $\P^{t,x}$ a.s. as $\beta \uparrow \infty$. We thus conclude from \eqref{eq:sqs1} and the monotone convergence theorem that
\[
u(t,x) = \lim_{\beta\to\infty} \E^{t,x}[X^{t,x}_T 1_{\{\tau^\beta \ge T\}}] = \E^{t,x}[X^{t,x}_T] = x,
\]
which implies the uniqueness of local stochastic solutions to \eqref{PDE X} in $\hat D$.

$(ii)\Leftarrow(iii)$: This is  trivial.

$(ii)\Rightarrow(iii)$: By Proposition~\ref{prop:E[T] is sol}, $u_1(t,x):=\E^{t,x}[g(X^{t,x}_T)]$ is a stochastic solution (and thus a local stochastic solution) to \eqref{PDE X g} in $\hat D$. By contradiction, assume that there exists a local stochastic solution $u_2$ to \eqref{PDE X g} in $\hat D$ such that $u_2\not\equiv u_1$. Define $\ell:=u_1-u_2\not\equiv 0$. It can be checked that $\ell\in\hat D$,
$
\ell(\cdot\wedge T, X^{t,x}_{\cdot\wedge T})=u_1(\cdot\wedge T,X^{t,x}_{\cdot\wedge T})-u_2(\cdot\wedge T, X^{t,x}_{\cdot\wedge T})
$
is a local martingale, $\ell(T,x)=u_1(T,x)-u_2(T,x)=0$ for $x\in[0,\infty)$, and $\ell(t,0)=u_1(t,0)-u_2(t,0)=0$ for all $t\in[0,T]$. Thus, by Proposition~\ref{prop:E[T] is sol} again, $\E^{t,x}[X^{t,x}_T]$ and $\E^{t,x}[X^{t,x}_T]+\ell(t,x)$ are distinct local stochastic solutions to \eqref{PDE X} in $\hat D$, which contradicts $(ii)$.
\end{proof}

\begin{remark}
In view of Propositions~\ref{prop:E[T] is sol} and \ref{prop:equiv}, $X$ is a martingale if and only if $U$ defined in \eqref{defn U} is the unique local stochastic solution to \eqref{PDE X g} in $\hat D$. A similar relation has been established in \cite[Proposition 5.4]{BKX12} in a stochastic volatility model, under additional regularity and growth condition on the coefficients of the associated state dynamics.
\end{remark}

We obtain the main result of this section, as a consequence of Theorem~\ref{thm:X true mart}, Propositions~\ref{prop:E[T] is sol} and \ref{prop:equiv}.

\begin{theorem}\label{thm:unique local stoch sol}
$U$ defined in \eqref{defn U} is the unique local stochastic solution to the Cauchy problem \eqref{PDE X g} in $\hat D$ for any $T>0$ and $g\in\hat D$ if and only if \eqref{DS02} holds.
\end{theorem}

Theorem~\ref{thm:unique local stoch sol} and Remark~\ref{p:class-stoch} yield a non-standard Feynman-Kac formula.

\begin{corollary}[Feynman-Kac formula]\label{coro:FK}
Suppose \eqref{DS02} holds. For any $T>0$ and $g\in\hat D$, if $u\in\hat D$ is a classical solution to \eqref{PDE X g}, then $u$ admits the stochastic representation $u(t,x)=\E^{t,x}[g(X^{t,x}_T)]$.
\end{corollary}

\begin{remark}
Note that the standing assumption \eqref{standing asm'} does not impose any continuity on $\sigma$, and $\eqref{DS02}$ is weaker than the linear growth condition on $\sigma$. This is in contrast to the usual conditions assumed on $\sigma$, continuity and linear growth, for a Feynman-Kac formula to hold (see e.g. \cite[Section 6.5]{Friedman-book-06} and \cite[Section 5.7.B]{KS-book-91}).
\end{remark}

\begin{remark}\label{rem:generalize Thm1}
Corollary~\ref{coro:FK} extends \cite[Thoerem 1]{BX10} to the case where $\sigma$ and $1/\sigma$ need not be locally bounded. 
Note that local boundedness of $\sigma$ and $1/\sigma$ is required in \cite[Thoerem 1]{BX10}, while it is not explicitly stated there. A convex function $\Psi$ is introduced in the proof of \cite[Theorem 1]{BX10}, and the authors, by referring to an argument in \cite[Theorem 1.6]{DS02}, states that  
\begin{equation}\label{ubb}
\E^{t,x}[\Psi(X^{t,x}_{\tau_n})]\quad \hbox{is uniformly bounded in $n$},
\end{equation}
where $\tau_n:=\inf\{s\ge t:X^{t,x}_s\notin(1/n,n)\}\wedge T$. Then, the de la Vall\'{e}e Poussin criterion can be applied to establish the uniform integrability of $\{X^{t,x}_{\tau_n}\}_{n\in\N}$, a crucial step in their proof. The argument in \cite[Theorem 1.6]{DS02}, however, requires local boundedness in $\sigma$ and $1/\sigma$. In particular, on p. 163 of \cite{DS02}, \eqref{ubb} is established under the condition $\E[\int_0^{\tau_n}\Psi'(X_{t})dX_{t}] =0$. In view of (1.1) and (1.2) in \cite{DS02}, this is equivalent to
\[
\E\left[\int_0^{\tau_n} \left(\int_1^{X_t} \frac{u}{\sigma^2(u)} du\right) \sigma(X_t)  dW_t \right]=0.
\] 
This condition holds when $\sigma$ and $1/\sigma$ are locally bounded, a standing assumption in \cite{DS02}, but may not be true for general $\sigma$. Since \cite[Theorem 1]{BX10} refers to the above argument in \cite{DS02}, it needs local boundedness of $\sigma$ and $1/\sigma$. 
\end{remark}

\begin{remark}
We focus on the state space $I=(0,\infty)$ because it is most relevant in terms of financial applications and corresponds to the pricing equation \eqref{PDE X g}. By the same arguments, one can actually extend all the developments to $I=(\ell,\infty)$ for some $\ell\in\R$, with $\hat D$ modified as the set of functions $u:[0,T]\times (\ell,\infty)\mapsto \R$ satisfying the linear growth condition, and \eqref{DS02} in Theorem~\ref{thm:unique local stoch sol} and Corollary~\ref{coro:FK} replaced by \eqref{DS02 l-finite}. 
\end{remark}


\section{Connection to Classical Solutions}\label{sec:classical}
In this section, we first derive a sufficient condition for the stochastic solution $U(t,x):=\E^{t,x}[g(X^{t,x}_T)]$ to be a classical solution to \eqref{PDE X g}. This problem has been studied in \cite{ET09}; here, we provide a weaker condition using a different approach. This, together with the uniqueness result in Section~\ref{sec:local stoch sol}, characterizes $U$ as the unique classical solution to \eqref{PDE X g}, under fairly general conditions. A comparison theorem for \eqref{PDE X g} is also established, without imposing linear growth condition on $\sigma$, as required in the standard literature. 

\subsection{Interior Smoothness of $U$} \label{sec:interior}
In \cite[Theorem 3.2]{ET09}, the authors showed that local H\"{o}lder continuity of $\sigma$ on $(0,\infty)$ with exponent $\delta\ge 1/2$ guarantees that $U$ is a classical solution to \eqref{PDE X g}. We will prove that the same result holds even with $\delta\in(0,1/2)$. 

This relaxation of H\"{o}lder continuity gives rise to technical difficulties. In \cite{ET09}, with $\delta\ge1/2$, \eqref{SDE X} admits a unique strong solution. Taking advantage of some continuity results of a strong solution with respect to initial data $(t,x)$, the authors proved that $U(t,x)=\E^{t,x}[g(X^{t,x}_T)]$ is continuous on compact domains. Since a continuous local stochastic solution is a classical solution (see Remark~\ref{rem:relations between sols}), this immediately yields interior smoothness of $U$.
By contrast, with $\delta\in (0,1/2)$, we have to work with weak solutions to \eqref{SDE X}, and the argument in \cite[Theorem 3.2]{ET09} does not work anymore. 
This makes it difficult to derive the continuity of $U(t,x)=\E^{t,x}[g(X^{t,x}_T)]$ a priori. 

In the following, we adopt a new strategy: we first construct a classical solution to \eqref{PDE X g} on each bounded domain by resorting to \cite{Lib96}, and show that those classical solutions converge to a smooth function $\hat u$ on the entire domain by an Arzela-Ascoli-type argument. Then, we show that $\hat u$ coincides with $U$ by probabilistic methods. 

\begin{lemma}\label{lem:smoothness}
Suppose $\sigma$ is locally H\"older continuous on $(0,\infty)$ with exponent $\delta\in(0,1]$. Then, for any $T>0$ and any continuous $g:[0,\infty)\mapsto \R$ in $\hat D$ that is bounded from below, the stochastic solution $U$, defined in \eqref{defn U}, belongs to $C^{1,2}([0,T)\times(0,\infty))\cap\hat D$ and satisfies \eqref{PDE X g}.
\end{lemma}

\begin{proof}
We will use the H\"{o}lder norm $|\cdot|_{2+\delta}$ and its weighted version $|\cdot|^*_{2+\delta}$ throughout the proof. We denote by $H_{k+\delta}$ (resp. $H^*_{k+\delta}$) the collection of functions with finite H\"{o}lder norms (resp. finite weighted H\"{o}lder norms). For readers' convenience, their definitions are given in Appendix~\ref{sec:Holder norms}.

We first assume that $g$ is bounded from below by $0$, and will deal with the general case in the last paragraph of the proof. Let $\{g_n\}$ be a sequence of nonnegative continuous functions such that $g_n\uparrow g$, with $g_n(x)=(1-\frac{1}{n})g(x)$ for $x\in[0,n-1/n]$ and $g_n(x)=0$ for $x\ge n$ (here, $g_n$ can be arbitrary on $(n-1/n,n)$, as long as $g_n$ is continuous and $g_n\uparrow g$). 
For any $n\in\N$, define $Q_{n} := [0,T) \times (1/n, n)$ and denote by $\partial^{*} Q_{n}$ the parabolic boundary of $Q_n$. Consider the following equation
\begin{equation}\label{PDEn}\tag{PDE$_n$}
\begin{cases}
\partial_{t} u + \frac 1 2 \sigma^{2} \partial_{xx} u = 0\ \ &\hbox{ in } Q_{n},\\
u(t,x) = g_{n}(x)\ \  &\hbox{ on } \partial^{*} Q_{n}.
\end{cases}
\end{equation}
Since $\sigma$ is H\"{o}lder continuous with exponent $\delta>0$ and $g_n$ is continuous, we deduce from \cite[Theorems 5.9 and 5.10]{Lib96} that there exists a classical solution $u_{n}\in H_{2 + \delta}^{*} (Q_{n}) \cap C(\bar Q_{n})$ to \eqref{PDEn}.

\begin{enumerate}
\item [1.] {\it There exists a $\hat u\in C^{1,2}([0,T)\times(0,\infty))$ such that $u_n\to\hat u$ on $[0,T)\times (0,\infty)$}. Take an increasing sequence $\{E_k\}$ of compact subsets of $[0,T)\times(0,\infty)$ such that $\bigcup_{k\in\N} E_k =[0,T)\times(0,\infty)$. Fix $k\in\N$. There exists some $N\in\N$ such that $E_k\subset Q_N$. Since $\sigma$ is locally H\"older continuous with exponent $\delta>0$, it in particular satisfies (4.20a) in \cite{Lib96} on the set $Q_N$. We may then apply \cite[Theorem 4.9]{Lib96}, which states that if $u\in H_{2 + \delta}^{*} (Q_{N})$ is a solution of 
\begin{equation}\label{eqn in Q_N}
\partial_{t} u + \frac 1 2 \sigma^{2} \partial_{xx} u = 0\ \ \hbox{ in } Q_{N},
\end{equation}
then there is a constant $K_1$, determined only by the function $\sigma$ and the set $Q_N$, such that
\[
|u|^{*}_{2+\delta, Q_N} \le K_1 |u|_{0, Q_N}.
\]
By construction, $u_n\in H_{2 + \delta}^{*} (Q_{N})$ satisfies \eqref{eqn in Q_N}, for all $n>N$. Thus,  
\begin{equation}
 \label{eq:sh1}
 |u_{n}|^{*}_{2+\delta, Q_N} \le K_1 |u_{n}|_{0, Q_N},\ \forall\ n> N.
\end{equation}
Then, by the definition of the weighted norm $|\cdot|^{*}_{2+\delta}$, there exists $K_2>0$ (independent of $n> N$) such that
\begin{equation}
\label{eq:sh2}
  |u_{n}|_{2+\delta, E_k} \le K_2 |u_{n}|^{*}_{2+\delta, Q_N}, \ \forall\ n>N.
\end{equation}
On the other hand, for each $n\in\N$, since $u_{n}$ is a smooth solution to \eqref{PDEn}, by It\^{o}'s rule we have the probabilistic representation
\begin{equation}\label{prob represent}
u_n(t,x)=\E^{t,x}[g_n(X^{t,x}_{\tau^n})]\ \ \ \hbox{for}\ (t,x)\in Q_n,\ \ \ \hbox{where}\ \tau^{n}:= \inf\{s\ge t: (s,X^{x,t}_{s})\notin Q_{n}\}.
\end{equation}
Considering that $0\le g_n\le g$, the linear growth condition of $g$ implies that there exists $K_3>0$ (independent of $n\in\N$) such that
\begin{equation*}
 \label{eq:sh3}
 |u_{n}(t,x)| \le \mathbb E^{t,x} [|g_{n}(X^{t,x}_{\tau^{n}})|]   \le \mathbb E^{t,x} [K_3(1+X^{t,x}_{\tau^{n}})]  \le K_3(1+x)\ \ \hbox{for}\ (t,x)\in Q_n,
\end{equation*}
which in particular implies that $|u_n|_{0,Q_N}\le K_3(1+N)$ for all $n > N$. This, together with \eqref{eq:sh1} and \eqref{eq:sh2}, shows that
\[
|u_{n}|_{2+\delta, E_k} \le K_4,
\]
for some $K_4>0$ (independent of $n > N$). Now, applying the Arzela-Ascoli theorem, we conclude that (up to a subsequence) $u_n$ converges uniformly on $E_k$ to some function $\hat{u}_k\in C^{1,2}(E_k)$. Since the above argument holds for each compact set $E_k$, we may choose a subsequence of $u_n$ (using the diagonal method) such that $u_n\to \hat u$ on $[0,T)\times (0,\infty)$ with $\hat u =\hat{u}_k$ on $E_k$ for all $k\in\N$. This in particular implies that $\hat u\in C^{1,2}([0,T)\times (0,\infty))$.

\item [2.] {\it $\hat u$ coincides with the stochastic solution $U$ on $[0,T)\times (0,\infty)$}. For any $(t,x)\in [0,T)\times(0,\infty)$, by \eqref{prob represent} and the fact that $g_n(n)=0$, we have
\begin{equation}\label{hat u=U}
\begin{split}
\hat u(t,x)&=\lim_{n\to\infty} u_n(t,x) = \lim_{n\to \infty}\E^{t,x}[g_n(X^{t,x}_{\tau^n})]\\
&=\lim_{n\to\infty}\E^{t,x}\left[g_n(X^{t,x}_T)1_{\{X^{t,x}_{\tau^n}\in(1/n,n)\}}+g_n(1/n)1_{\{X^{t,x}_{\tau^n}=1/n\}}\right]\\
&=\lim_{n\to\infty}\E^{t,x}\left[g_n(X^{t,x}_T)1_{\{X^{t,x}_{\tau^n}\in[1/n,n)\}}+[g_n(1/n)-g_n(X^{t,x}_T)]1_{\{X^{t,x}_{\tau^n}=1/n\}}\right].
\end{split}
\end{equation}
Let $\tau^0:=\inf\{s\ge t:X^{t,x}_s=0\}\wedge T$. On the set $\{X^{t,x}_{\tau^0}>0\}$, $X^{t,x}_{\tau^n}>1/n$ for $n$ large enough $\P^{t,x}$-a.s. Since $|g_n(1/n)-g_n(X^{t,x}_T)|\le \sup_{y \in [0,1]} g(y)+g(X^{t,x}_T)$ for all $n\in\N$ and $g(X^{t,x}_T)$ is integrable as $g\in\hat D$, by using the definition of $g_n$ and the dominated convergence theorem,
\begin{equation}\label{tau^0>0}
\lim_{n\to\infty}\E^{t,x}\left[[g_n(1/n)-g_n(X^{t,x}_T)]1_{\{X^{t,x}_{\tau^n}=1/n\}}1_{\{X^{t,x}_{\tau^0}>0\}}\right]=0.
\end{equation}
On the other hand, on the set $\{X^{t,x}_{\tau^0}=0\}$, 
\begin{equation}\label{tau^0=0}
\begin{split}
\lim_{n\to\infty}&\E^{t,x}\left[[g_n(1/n)-g_n(X^{t,x}_T)]1_{\{X^{t,x}_{\tau^n}=1/n\}}1_{\{X^{t,x}_{\tau^0}=0\}}\right]\\
&=\lim_{n\to\infty}\E^{t,x}\left[[(1-1/n)g(1/n)-g_n(0)]1_{\{X^{t,x}_{\tau^n}=1/n\}}1_{\{X^{t,x}_{\tau^0}=0\}}\right] = 0,
\end{split}
\end{equation}
where the first equality follows from the definition of $g_n$ and the fact that $X^{t,x}$ is absorbed at zero, and the second equality is due to the dominated convergence theorem and $\lim_{n\to\infty}(1-1/n)g(1/n)-g_n(0)=0$, thanks again to the definition of $g_n$. Now, combining \eqref{hat u=U}, \eqref{tau^0>0}, and \eqref{tau^0=0}, we obtain
\[
\hat u(t,x)=\lim_{n\to\infty}\E^{t,x}\left[g_n(X^{t,x}_T)1_{\{X^{t,x}_{\tau^n}\in[1/n,n)\}}\right].
\]
By the monotone convergence theorem, $\hat u(t,x)= \E^{t,x}[g(X^{t,x}_T)]=U(t,x)$.
\end{enumerate}
Finally, if $g$ is bounded from below by some $\ell<0$, then $\widetilde g:= g-\ell$ is bounded from below by $0$. The proof above immediately implies that $\widetilde U(t,x) := \E^{t,x}[\widetilde g(X^{t,x}_T)] = \E^{t,x}[g(X^{t,x}_T)] -\ell$ belongs to $C^{1,2}([0,T)\times(0,\infty))\cap\hat D$ and satisfies \eqref{PDE X g} with boundary conditions given by $\widetilde g$. Thus, $U(t,x)=\E^{t,x}[g(X^{t,x}_T)]$ belongs to $C^{1,2}([0,T)\times(0,\infty))\cap\hat D$ and satisfies \eqref{PDE X g}.
\end{proof}


\subsection{Continuity of $U$ up to the Boundary} \label{subsec:continuity at bdd}
In \cite[Theorem 3.2]{ET09}, interior smoothness of $U$ and the continuity of $U $ up to the boundary are obtained through a monotone smooth approximation. In our case with weaker H\"{o}lder continuity of $\sigma$, we treat interior smoothness of $U$ and its continuity up to the boundary separately. As shown in Lemma~\ref{lem:smoothness}, we derive the interior smoothness via a different smooth approximation (which may not be monotone). To establish the continuity up to the boundary, we will resort to tools from viscosity solutions developed in \cite{BS12}, which do not rely on any smoothness (or smooth approximation) of $U$. 

Let us first recall the notation and results in \cite{BS12}. Analogous to Definition~\ref{defn:stoch sol X}, we say a Borel function $u:[0,T]\times [0,\infty) \mapsto\R$ is a stochastic subsolution (resp. stochastic supersolution) to \eqref{PDE X g} if
\begin{itemize}
\item [(i)]  $u\left(r\wedge T,X^{t,x}_{r\wedge T}\right)$ is a submartingale (resp. supermartingale), for any weak solution to \eqref{SDE X} with initial condition $(t,x)\in[0,T]\times \R$.
\item [(ii)] for any $x\in (0, \infty)$, $u(T,x)\le g(x)$ (resp. $u(T,x)\ge g(x)$); for any $t\in(0,T]$, $u(t,0)\le g(0)$ (resp. $u(t,0)\ge g(0)$).
\end{itemize}
For any $g:[0,\infty)\mapsto\R$, we denote by $\mathcal{U}^-_g$ (resp. $\mathcal{U}^+_g$) the collection of all lower semicontinuous stochastic subsolutions (resp. upper semicontinuous stochastic supersolutions) to \eqref{PDE X g}. 
When $g$ is bounded from below, the collection $\mathcal{U}^-_g$ is nonempty. Indeed, for any $c\in\R$ such that $g(x)\ge c$ for all $x\in[0,\infty)$, $u(t,x)\equiv c$ lies in $\mathcal{U}^-_g$. For any $u\in\mathcal{U}^-_g$, $u(t,x)\le \E^{t,x}[g(X^{t,x}_T)]$ by definition. It follows that
\begin{equation}\label{v^-<U}
v^-_g(t,x):=\sup_{u\in\mathcal{U}^-_g}u(t,x)\le\E^{t,x}[g(X^{t,x}_T)].
\end{equation}
Note that $v^-_g$ is by definition lower semicontinuous. If $g$ is additionally lower semicontinuous, we may apply the same argument in \cite[Theorem 2.1]{BS12} to show that
\begin{equation}\label{v^->g}
v^-_g(T,x)\ge g(x).
\end{equation}

The next result shows that $U$ is continuous up to the parabolic boundary of $[0,T)\times(0,\infty)$. It is worth noting that no regularity of $\sigma$, not even continuity, is required below.

\begin{lemma}\label{lem:p-boundary condition}
For any $T>0$ and any continuous $g:[0,\infty)\mapsto\R$ in $\hat D$ that is bounded from below, the stochastic solution $U$, defined in \eqref{defn U}, satisfies:
\begin{equation}\label{p-boundary condition}
\begin{split}
U^*(T,x)=U_*(T,x)=U(T,x)=g(x)\ \  &\hbox{for}\ x\in(0,\infty),\\
U^*(t,0)=U_*(t,0)=U(t,0)=g(0)\ \  &\hbox{for}\ t\in[0,T],
\end{split}
\end{equation} where $U^*$ and $U_*$ are the upper and lower semicontinuous envelopes of $U$: 
\begin{align*}
U^{*}(t,x) &: = \lim_{\delta\downarrow 0}\sup_{} \{U(s,y): s\in ((t-\delta)\vee 0,t],\ y\in((x-\delta)\vee 0,x+\delta)\},\\
U_{*}(t,x) &: = \lim_{\delta\downarrow 0}\inf\{U(s,y): s\in ((t-\delta)\vee 0,t],\ y\in((x-\delta)\vee 0,x+\delta)\}.
\end{align*}
\end{lemma}

\begin{proof}
We will prove the desired result first for $g$ satisfying some additional properties, and then extend it to the general case.
\begin{itemize}
\item [1.] {\it $g$ is concave.} Since $g$ is bounded from below, $\mathcal{U}^-_g$ is nonempty. Moreover, being bounded from below and concave implies that $g$ is nondecreasing. In view of \eqref{v^-<U} and \eqref{defn U}, 
\[
v^-_g(t,x)\le U(t,x)\le g(\E^{t,x}[X^{t,x}_T])\le g(x),
\]
where the second inequality is due to Jensen's inequality, and the third inequality follows from $X^{t,x}_\cdot$ being a supermartingale and $g$ being nondecreasing. 
By the continuity of $g$, the lower semicontinuity of $v^-_g$ and \eqref{v^->g}, we have $U^*(T,x)\le g(x)$ and $U_*(T,x)\ge v^-_g(T,x)\ge g(x)$. Since $U(T,x)=g(x)$ by definition, we conclude that $U^*(T,x)=U_*(T,x)=g(x)=U(T,x)$. On the other hand, since $g$ is nondecreasing, one must have
\[
0\le\E^{t,x}[g(X^{t,x}_T)-g(0)]=U(t,x)-g(0)\le g(x)-g(0).
\]
Then the continuity of $g$ yields $U^*(t,0)\le g(0)$ and $U_*(t,0)\ge g(0)$. Since $U(t,0)=g(0)$ by definition, we conclude that $U^*(t,0)=U_*(t,0)=g(0)=U(t,0)$.

\item [2.] {\it $g$ is a smooth function satisfying}
$\int_0^\infty|g''(y)|dy<\infty$.
As observed in \cite[p.1375]{ET09}, $g=g_1-g_2$ for some nonnegative concave functions $g_1$ and $g_2$. By Step 1, $U_1(t,x):=\E^{t,x}[g_1(X^{t,x}_T)]$ and $U_2(t,x):=\E^{t,x}[g_2(X^{t,x}_T)]$ both satisfy \eqref{p-boundary condition}. Also, by the definition of semicontinuous envelopes,
\begin{align*}
U^*(t,x)&=(U_1-U_2)^*(t,x)\le U_1^*(t,x)-(U_2)_*(t,x),\\
U_*(t,x)&=(U_1-U_2)_*(t,x)\ge (U_1)_*(t,x)-U_2^*(t,x).
\end{align*}
This implies that $U^*(T,x)\le g_1(x)-g_2(x)=g(x)$, $U^*(t,0)\le g_1(0)-g_2(0)=g(0)$, $U_*(T,x)\ge g_1(x)-g_2(x)=g(x)$, and $U_*(t,0)\le g_1(0)-g_2(0)=g(0)$. Since $U(T,x)=g(x)$ and $U(t,0)=g(0)$ by definition, we conclude that $U$ satisfies \eqref{p-boundary condition}.

\item [3.] {\it The general case: $g$ is continuous.} We can take two sequences $\{g_n\}$ and $\{g^n\}$ of smooth functions satisfying $\int_0^\infty|g''(y)|dy<\infty$ such that $g_n\uparrow g$ and $g^n\downarrow g$. By Step 2, $U_n(t,x):=\E^{t,x}[g_n(X^{t,x}_T)]$ and $U^n(t,x):=\E^{t,x}[g^n(X^{t,x}_T)]$ both satisfy \eqref{p-boundary condition}, for all $n\in\N$. Then $U=\lim_{n\to\infty}U_n=\sup_n U_n$ and $U=\lim_{n\to\infty}U^n=\inf_n U^n$, thanks to the theorems of monotone convergence and dominated convergence, respectively. Noting that $\inf_n (U^n)^*$ is upper semicontinuous and $\sup_n (U^n)_*$ is lower semicontinuous, we have
\begin{equation*}
U^*(t,x)=(\inf_n U^n)^*(t,x)\le \inf_n (U^n)^*(t,x),\quad U_*(t,x)=(\sup_n U_n)_*(t,x)\ge \sup_n (U_n)_*(t,x).
\end{equation*}
Therefore, $U^*(T,x)\le \inf_n (U^n)^*(T,x)=\inf_n g^n(x)=g(x)$ and $U^*(t,0)\le \inf_n (U^n)^*(t,0)=\inf_n g^n(0)=g(0)$. Similarly, $U_*(T,x)\ge \sup_n (U_n)_*(T,x)=\sup_n g_n(x)=g(x)$ and $U_*(t,0)\ge \sup_n (U_n)_*(t,0)=\sup_n g_n(0)=g(0)$. Since $U(T,x)=g(x)$ and $U(t,0)=g(0)$ by definition, we conclude that $U$ satisfies \eqref{p-boundary condition}.
\end{itemize}
\end{proof}

\begin{remark}
To the best of our knowledge, our method of deriving continuity up to the boundary, which is based on the stochastic Perron's method in \cite{BS12},  is new in the literature.
\end{remark}

Lemmas~\ref{lem:smoothness} and~\ref{lem:p-boundary condition} together lead to:

\begin{theorem}\label{thm:E[] is classical}
Suppose $\sigma$ is locally H\"older continuous on $(0,\infty)$ with exponent $\delta\in(0,1]$. For any $T>0$ and any continuous $g:[0,\infty)\mapsto\R$ in $\hat D$ that is bounded from below, the stochastic solution $U$, defined in \eqref{defn U}, is a classical solution to \eqref{PDE X g} in $\hat D$.
\end{theorem}

The main result of this section is the following characterization of the stochastic solution $U$. It in particular generalizes \cite[Theorem 2]{BX10}, as it requires less H\"older continuity of $\sigma$.

\begin{theorem}\label{thm:exist and unique}
Suppose $\sigma$ is locally H\"older continuous on $(0,\infty)$ with exponent $\delta\in(0,1]$. Then, \eqref{DS02} holds if and only if the stochastic solution $U$, defined in \eqref{defn U}, is the unique classical solution to \eqref{PDE X g} in $\hat D$, for any $T>0$ and any continuous $g:[0,\infty)\mapsto\R$ in $\hat D$ that is bounded from below.
\end{theorem}

\begin{proof}
By Theorem~\ref{thm:E[] is classical}, $U(t,x)=\E^{t,x}[g(X^{t,x}_T)]$ is a classical solution to \eqref{PDE X g} in $\hat D$, for any $T>0$ and any continuous $g:[0,\infty)\mapsto\R$ in $\hat D$ that is bounded from below. Then the necessity follows immediately from Corollary~\ref{coro:FK}. To prove the sufficiency, assume to the contrary that \eqref{DS02} is not true. By Proposition~\ref{prop:equiv}, $X^{t,x}$ must be a strict local martingale, and thus there exists $T^*>0$ such that $\E^{t,x}[X^{t,x}_{T^*}]<x$. Using Theorem~\ref{thm:E[] is classical} again, $\E^{t,x}[X^{t,x}_{T^*}]$ is a classical solution to \eqref{PDE X g} with $g(x):=x$ and $T:=T^*$. Then $x$ and $\E^{t,x}[X^{t,x}_{T^*}]$ are two distinct classical solution to \eqref{PDE X g} with $g(x):=x$ and $T:=T^*$, a contradiction.
\end{proof}

\begin{remark}\label{rem:compare KR13}
Concerning the stochastic representation of classical solutions, Theorem~\ref{thm:exist and unique} above and \cite[Proposition 5.2]{Karatzas-Ruf-13} share similar spirit, yet under quite different scenarios. In \cite{Karatzas-Ruf-13}, since the value function is the probability of explosion for a general diffusion, the authors focus on {\it bounded} solutions to a Cauchy problem with a non-zero drift. Here, our main concern is pricing European contingent claims under some risk-neutral measure. We therefore focus on {\it unbounded} solutions to a Cauchy problem with a zero drift.
\end{remark}


\subsection{A Comparison Theorem without Linear Growth Condition on $\sigma$} 
A comparison theorem states that if a supersolution $v$ is larger than a subsolution $u$ at the (parabolic) boundary, then $v\ge u$ on the entire domain. For such a theorem to hold, linear growth condition on $\sigma$ is a standard assumption, which ensures that the classical penalization method works; see e.g. \cite[p.76]{Pham-book-09}. In the following, we will replace the linear growth condition on $\sigma$ by the more general condition \eqref{DS02}. Our main contribution is the construction of a smooth function $\psi(x)$ in \eqref{psi for comparison} below, which is tailor-made so that a penalization can still work under \eqref{DS02}.

\begin{proposition}\label{prop:comparison}
Suppose $\sigma$ is continuous and satisfies \eqref{DS02}. Let $u\in\hat D$ (resp. $v\in\hat D$) be a classical subsolution (resp. supersolution) to
\[
-\partial_t w -\frac{1}{2} \sigma^2(x) \partial_{xx}w =0\ \ \hbox{on}\ [0,T)\times (0,\infty).
\]
If $u(T,x)\le v(T,x)$ for all $x\in(0,\infty)$ and $u(t,0)\le v(t,0)$ for all $t\in[0,T]$, then $u\le v$ on $[0,T]\times [0,\infty)$.
\end{proposition}

\begin{proof}
For $\lambda>0$, define $\tilde{u}:=e^{\lambda t}u$ and $\tilde{v}:=e^{\lambda t}v$. Clearly, $\tilde{u}$ and $\tilde{v}$ lie in $\hat D$, and $\tilde{u}$ (resp. $\tilde{v}$) is a classical subsolution (resp. supersolution) to
\begin{equation}\label{PDE lambda}
-\partial_t w+\lambda w-\frac{1}{2} \sigma^2(x) \partial_{xx}w =0.
\end{equation}
Consider the function $\phi(t,x):=e^{-\lambda t}\psi(x)$, where
\begin{equation}\label{psi for comparison}
\psi(x):=
\begin{cases}
x\ \ &\hbox{if}\ x\le 1,\\
x+\int_1^x \frac{(u-1)(x-u)}{\sigma^2(u)}du\ \ &\hbox{if}\ x>1.
\end{cases}
\end{equation}
It can be checked by direct calculations that $\psi$ is twice-differentiable on $(0,\infty)$, with $\psi''(x)$ equal to either $0$ when $x\le 1$, or $\frac{x-1}{\sigma^2(x)}$ when $x>1$. This, together with $\psi(x)\ge x$ by definition, implies
\[
-\partial_t\phi+\lambda \phi-\frac{1}{2} \sigma^2(x) \partial_{xx}\phi=e^{-\lambda t}\left(2\lambda\psi(x)-\frac{1}{2}\sigma^2(x)\psi''(x)\right)\ge 0,\ \ \ \hbox{for}\ \lambda\ge 1/4.
\]
It follows that, by fixing $\lambda\ge 1/4$, $\tilde{v}_\eps:=\tilde{v}+\eps\phi$ is a supersolution to \eqref{PDE lambda}, for all $\eps>0$.

On the other hand, observe that \eqref{DS02} implies $\int_1^\infty \frac{u-1}{\sigma^2(u)}du=\infty$ (as a simple application of L'H\^{o}pital's rule). 
We therefore have
\[
\lim_{x\to\infty}\frac{\eps\phi(t,x)}{x}=\lim_{x\to\infty}\eps e^{-\lambda t}\left(1+\int_1^x \frac{u-1}{\sigma^2(u)}du\right)=\infty\ \ \ \hbox{for any}\ \eps>0.
\]
We then deduce from this and the linear growth condition on $\tilde{u}$ and $\tilde{v}$ that
\begin{equation}\label{max at interior}
\lim_{x\to\infty}\sup_{[0,T]} (\tilde{u}-\tilde{v}_\eps)(t,x)=-\infty.
\end{equation}

Now, we assume that there exists some $\eps>0$ such that $\tilde{u}>\tilde{v}_\eps$ at some point in $[0,T]\times [0,\infty)$, and will work toward a contradiction. In view of \eqref{max at interior}, and the boundary and terminal conditions of $u$ and $v$, there must exist some $(\bar{t},\bar{x})\in[0,T)\times (0,\infty)$ such that
\begin{equation}\label{to contradict}
0<\sup_{[0,T]\times[0,\infty)}(\tilde{u}-\tilde{v}_\eps)=(\tilde{u}-\tilde{v}_\eps)(\bar{t},\bar{x}).
\end{equation}
It then follows from the first and second-order optimality conditions that
\[
[\partial_t\tilde{u}-\partial_t\tilde{v}_\eps](\bar{t},\bar{x})\le 0,\ \ [\partial_{xx}\tilde{u}-\partial_{xx}\tilde{v}_\eps](\bar{t},\bar{x})\le 0.
\]
Now, using the above inequalities and the fact that $\tilde{u}$ (resp. $\tilde{v}_\eps$) is a subsolution (resp. supersolution) to \eqref{PDE lambda}, we obtain
\begin{align*}
\lambda(\tilde{u}-\tilde{v}_\eps)(\bar{t},\bar{x})\le [\partial_t\tilde{u}-\partial_t\tilde{v}_\eps](\bar{t},\bar{x})+\frac{1}{2}\sigma^2(\bar{x})[\partial_{xx}\tilde{u}-\partial_{xx}\tilde{v}_\eps](\bar{t},\bar{x})\le 0,
\end{align*}
which contradicts \eqref{to contradict}.
\end{proof}

\begin{remark}
The function $\psi$ in \eqref{psi for comparison} also appeared, in a slightly different form, in the proofs of \cite[Theorem 1]{BX10} and \cite[Theorem 1.6]{DS02}. In both places, the function was used to prove certain uniform integrability property related to the process $X^{t,x}_\cdot$. Our contribution here is observing that this function is also instrumental to establishing a comparison theorem.
\end{remark}

By Theorems~\ref{thm:X true mart} and~\ref{thm:exist and unique}  and Proposition~\ref{prop:comparison}, we have the following:

\begin{theorem}\label{thm:full equiv}
Suppose $\sigma$ is locally H\"{o}lder continuous on $(0,\infty)$ with exponent $\delta\in(0,1]$. Then the following are equivalent:
\begin{itemize}
\item [(i)] \eqref{DS02} is satisfied.
\item [(ii)] $X$ is a true martingale.
\item [(iii)] For any $T>0$ and any continuous $g:[0,\infty)\mapsto\R$ in $\hat D$ that is bounded form below, \eqref{PDE X g} admits a unique classical solution in $\hat D$.
\item [(iv)] For any $T>0$ and any continuous $g:[0,\infty)\mapsto\R$ in $\hat D$ that is bounded form below, a comparison theorem for \eqref{PDE X g} holds among sub(super-)solutions in $\hat D$.
\end{itemize}
\end{theorem}

\begin{remark}
While it is obvious that a comparison theorem implies the uniqueness of solutions, the converse in general may not be true. These two concepts, however, are equivalent in our case, as shown by ``(iii) $\Leftrightarrow$ (iv)'' in Theorem~\ref{thm:full equiv}.
\end{remark}

\begin{remark}
In a stochastic volatility model, a relation similar to ``(ii) $\Leftrightarrow$ (iii)'' in Theorem~\ref{thm:full equiv} has been established in \cite[Theorems 2.8 and 2.9 (ii)]{BKX12}, under additional regularity and growth condition on the coefficients of the associated state dynamics.
\end{remark}


\section{Characterizing the Stochastic Solution without Smoothness}\label{sec:U as a limit}

In this section, we will work under the assumption that $\sigma$ is continuous, but may not be locally H\"{o}lder continuous on $(0,\infty)$ with a fixed exponent $\delta\in(0,1)$. 
Therefore we can not use Lemma~\ref{lem:smoothness} to show that the stochastic solution $U(t,x):=\E^{t,x}[g(X^{t,x}_T)]$ is smooth and hence the characterization for $U$ provided in Theorem~\ref{thm:exist and unique} is not applicable. Our goal is to find a new characterization for $U$ which relies on continuity of $\sigma$ only.


\subsection{Discussion on Viscosity Characterization of $U$ and Stochastic Perron's Method} In view of the characterization of $U$ in Theorem~\ref{thm:E[] is classical}, it is natural to ask whether $U$ can be characterized as the unique viscosity solution to \eqref{PDE X g} when $U$ may fail to be smooth.

To derive the interior viscosity solution property, some form of dynamic programming principle is usually needed. Since \eqref{PDE X g} is a linear equation, the dynamic programming principle reduces to the strong Markov property, which is inherent in Definition~\ref{defn:stoch sol X}. Indeed, under the assumption that $\sigma$ is continuous, one may deduce from Definition~\ref{defn:stoch sol X} that $U^*$ (resp. $U_*$) is a viscosity subsolution (resp. supersolution) to
\begin{equation}\label{vis X g}
-\partial_t u-\frac{1}{2}\sigma^2(x)\partial_{xx}u=0,\ \ (t,x)\in[0,T)\times (0,\infty).
\end{equation}
This, together with Lemma~\ref{lem:p-boundary condition}, already shows that
\begin{itemize}
\item [] $U^*$ is a viscosity subsolution to \eqref{vis X g}, with $U^*(T,\cdot)=g(\cdot)$ and $U^*(\cdot,0)=g(0)$.
\item [] $U_*$ is a viscosity supersolution to \eqref{vis X g}, with $U_*(T,\cdot)=g(\cdot)$ and $U_*(\cdot,0)=g(0)$.
\end{itemize}
Deriving a viscosity comparison theorem for \eqref{PDE X g} is then the final step for characterizing $U$ as the unique viscosity solution to \eqref{PDE X g}. 

Another interesting approach is stochastic Perron's method introduced in \cite{BS12}. The authors discover that, when $g$ is bounded, the function $v^-_g(t,x)$ in \eqref{v^-<U} and $v^+_g(t,x):= \inf\{u(t,x):u\in\mathcal{U}^+_g\}$ are well-defined, and satisfy the viscosity subsolution property and viscosity supersolution property, respectively, for the equation \eqref{vis X g}; recall the notation in Subsection~\ref{subsec:continuity at bdd}. By construction, $v^-_g\le \E^{t,x}[g(X^{t,x}_T)]\le v^+_g$ holds for any weak solution to \eqref{SDE X}. As a result, a viscosity comparison theorem for \eqref{PDE X g} again characterizes $U$ as the unique viscosity solution to \eqref{PDE X g}. Note that this approach does not even require a priori the uniqueness in distribution of weak solutions to \eqref{SDE X}.

While a viscosity comparison theorem can indeed be proved, we notice that it does not serve our needs. 
To derive a viscosity comparison theorem, a local Lipschitz condition on $\sigma$ is indispensable (this is in contrast to deriving a comparison theorem for classical solutions); see e.g. \cite[Section 4.4]{Pham-book-09}. Under the local Lipschitz condition on $\sigma$ and \eqref{DS02}, one can establish a viscosity comparison theorem by using the penalization in the proof of Proposition~\ref{prop:comparison}, a dedoubling variable technique (see e.g. \cite[p.78]{Pham-book-09}), and Ishii's lemma (see e.g. \cite[Lemma 4.4.6 and Remark 4.4.9]{Pham-book-09}). However, a local Lipschitz condition on $\sigma$ already implies that $U$ is smooth (by Theorem~\ref{thm:E[] is classical}), and we can thus characterize $U$ as the unique classical solution to \eqref{PDE X g} under \eqref{DS02}, as in Theorem~\ref{thm:exist and unique}. Deriving a viscosity comparison theorem is then superfluous. In view of this, we choose not to pursue viscosity characterization any further. 


\subsection{Characterize $U$ as a Limit of Smooth Stochastic Solutions}
Assume that $\sigma$ is a positive continuous function defined on $(0,\infty)$. We approximate $\sigma$ by a sequence $\{\sigma_n\}$ of positive continuous functions such that
\begin{equation}\label{sigma_n-sigma}
\begin{split}
\hbox {(i)}&\ \sigma_n\ \hbox{is locally H\"{o}lder continuous on $(0,\infty)$ with some exponent}\ \delta_n\in(0,1];\\
\hbox{(ii)}&\ \sigma_n\uparrow\sigma;\\
\hbox{(iii)}&\ \hbox{for any compact}\ K\subset (0,\infty),\ \max_{x\in K}\left\{\frac{1}{\sigma_n^2(x)}-\frac{1}{\sigma^2(x)}\right\}<\frac{1}{n}.
\end{split}
\end{equation}
Such an approximating sequence $\{\sigma_n\}$ does exist. For example, by using \cite[Theorem A]{GS92}, we may take a sequence $\{f_n\}$ of piecewise polynomials defined on $(0,\infty)$ such that $f_n\downarrow \frac{1}{\sigma^2}$ and $\max_{x\in K}\left\{f_n(x)-\frac{1}{\sigma^2(x)}\right\}<\frac{1}{n}$ for any compact set $K\subset (0,\infty)$. Then it can be checked that $\sigma_n:= 1/\sqrt{f_n}$ 
satisfies all the conditions in \eqref{sigma_n-sigma}.

Let $X^{(n),t,x}$ denote the unique (in distribution) weak solution to \eqref{SDE X}, with the diffusion coefficient $\sigma_n$. We have the following convergence result.

\begin{lemma}\label{lem:X^n to X}
Suppose $\sigma$ is a positive continuous function defined on $(0,\infty)$. For any $T>0$ and $(t,x)\in[0,T]\times[0,\infty)$, $X^{(n),t,x}_T\to X^{t,x}_T$ in distribution.
\end{lemma}

\begin{proof}
Let $B^{t,x}$ be a one-dimensional Brownian motion defined on some filtered probability space $(\Omega,\F,\{\F_s\}_{s\ge t},\P)$, starting with $B^{t,x}_t=x>0$. Let us extend the domain of $\sigma$ (resp. $\sigma_n$) to $\R$ by setting $\sigma(x)=0$ (resp. $\sigma_n(x)=0$) for all $x\le 0$. Define
\[
\Pi_s:=\int_t^s\frac{1}{\sigma^2(B^{t,x}_u)}du,\quad \Pi^{(n)}_s:=\int_t^s\frac{1}{\sigma_n^2(B^{t,x}_u)}du,\quad s\ge t,
\]
and consider
\begin{equation}\label{A's}
\begin{split}
& A_s:=\inf\left\{r\ge t: \Pi_r\ge s\right\},\quad A^{(n)}_s:=\inf\{r\ge t: \Pi^{(n)}_r\ge s\},\quad s\ge t;\\
& A_\infty:=\inf\left\{r\ge t: \Pi_r= \infty\right\},\quad A^{(n)}_\infty:=\inf\{r\ge t: \Pi^{(n)}_r= \infty\}.
\end{split}
\end{equation}
Define $\tau:=\inf\{s\ge t: B^{t,x}_s=0\}$. Since local integrability of $\sigma$ and $\sigma_n$ holds on $(0,\infty)$, but fails at $x=0$, \cite[Lemma 5.5.2]{KS-book-91} asserts that $\tau=A_\infty = A^{(n)}_\infty$ $\P$-a.s. Also, in view of the proof of \cite[Theorem 5.5.4]{KS-book-91}, $\{B^{t,x}_{A_s}\}_{s\ge t}$ (resp. $\{B^{t,x}_{A^{(n)}_s}\}_{s\ge t}$) is a weak solution to \eqref{SDE X} (resp. \eqref{SDE X} with the diffusion coefficient $\sigma_n$) with $I=(0,\infty)$.

Since $\sigma_{n } \uparrow \sigma$ on $(0,\infty)$ and $A^{(n)}_T \le A^{(n)}_\infty =\tau$ $\P$-a.s. for all $n\in \N$, we deduce from \eqref{A's} that $A^{(n)}_{T} \le A^{(n+1)}_{T}  \le A_{T}$ $\P$-a.s. for all $n\in\N$. This implies that the limit 
\begin{equation}\label{A^infty_T}
A^{(\infty)}_T :=\ \uparrow \lim_{n\to\infty} A^{(n)}_{T}\le A_T\quad \P\hbox{-a.s.}
\end{equation}
We claim that $A^{(\infty)}_T = A_T$ $\P$-a.s. In the following, we deal with two different cases separately. For simplicity, we will omit the $\omega$-dependence  in the random variables to be used  in the sequel.

\begin{enumerate}
\item[1.] {\it There exists $N\in\N$ such that $A^{(N)}_T=\tau$.} Since $A^{(n)}_T\le A^{(n)}_\infty=\tau$ for all $n\in\N$ and $n\mapsto A^{(n)}_T$ is nondecreasing, we must have $A^{(n)}_T=\tau$ for all $n\ge N$, and thus $A^{(\infty)}_T=\tau$. Recalling that $A_T\le A_\infty=\tau$ and $A_T\ge A^{(\infty)}_T$ in \eqref{A^infty_T}, we conclude that $A_T=\tau=A^{(\infty)}_T$.
\item[2.] {\it $A^{(n)}_T<\tau$ for all $n\in\N$.} Then, we deduce from \eqref{A's} that 
\begin{equation}\label{T=}
T=\Pi^{(n)}_{A^{(n)}_{T}} = \int_t^{A^{(n)}_T}\frac{1}{\sigma_n^{2}(B^{t,x}_u)}du\quad \hbox{for all}\ n\in\N. 
\end{equation}
Moreover, for each $n\in\N$, the interval
\[
K_n:=\bigg[\min_{u\in[t,A^{(n)}_T]} B^{t,x}_u, \max_{u\in[t,A^{(n)}_T]} B^{t,x}_u\bigg]
\]
is a compact subset of $(0,\infty)$. Thus, by condition (iii) in \eqref{sigma_n-sigma} (applying to $K_n$), 
\begin{align*}
T= \int_t^{A^{(n)}_T}\frac{1}{\sigma_n^2(B^{t,x}_u)}du &\le \int_t^{A^{(n)}_T}\frac{1}{\sigma^2(B^{t,x}_u)}du+\frac{1}{n}(A^{(n)}_T-t)\\
&\le \int_t^{A^{(\infty)}_T}\frac{1}{\sigma^2(B^{t,x}_u)}du+\frac{1}{n}(A^{(\infty)}_T-t).
\end{align*}
Since $A^{(\infty)}_T\le A_{T}\le\tau$ and $\tau$ is by definition finite $\P$-a.s. (as a basic property of a Brownian motions), letting $n\to\infty$ in the above inequality gives $T\le \int_t^{A^{(\infty)}_T}\frac{1}{\sigma^{2}(B^{t,x}_u)}du$. 
Using $A_T\le A_\infty = \tau$ and \eqref{A's} as in \eqref{T=}, we get $T= \Pi_{A_{T}} = \int_t^{A_T}\frac{1}{\sigma^{2}(B^{t,x}_u)}du$. We therefore obtain
\[
\int_t^{A_T}\frac{1}{\sigma^2(B^{t,x}_u)}du\le T\le \int_t^{A^{(\infty)}_T}\frac{1}{\sigma^2(B^{t,x}_u)}du,
\]
which implies $A_T\le A^{(\infty)}_T$. 
\end{enumerate}
We conclude that $A_T= A^{(\infty)}_T$ $\P$-a.s., and thus $B^{t,x}_{A^{(n)}_T} \to B^{t,x}_{A_T}$ $\P$-a.s. as $n \to \infty$.
Let us now complete the proof. Recall that $\{B^{t,x}_{A_s}\}_{s\ge t}$ (resp. $\{B^{t,x}_{A^{(n)}_s}\}_{s\ge t}$) is a weak solution to \eqref{SDE X} (resp. \eqref{SDE X} with the diffusion coefficient $\sigma_n$). Uniqueness in distribution of weak solutions imply that
\begin{equation*}
\hbox{Law}(X^{t,x}_T)=\hbox{Law}(B^{t,x}_{A_T}),\ \ \ \hbox{Law}(X^{(n),t,x}_T)=\hbox{Law}(B^{t,x}_{A^{(n)}_T}).
\end{equation*}
This, together with the almost surely convergence of $B^{t,x}_{A^{(n)}_T}$ to $B^{t,x}_{A_T}$, shows that $X^{(n),t,x}_T$ converges to $X^{t,x}_T$ in distribution.
\end{proof}

\begin{theorem}\label{thm:U as a limit}
Let $\sigma$ be a positive continuous function defined on $(0,\infty)$, and $\{\sigma_n\}$ be specified as in \eqref{sigma_n-sigma}. For any $T>0$ and any continuous $g:[0,\infty)\mapsto\R$ in $\hat D$ that is bounded from below, define $U(t,x):=\E^{t,x}[g(X^{t,x}_T)]$ and $U_n^M(t,x):=\E^{t,x}[g(X^{(n),t,x}_T)\wedge M]$ for all $n,M\in\N$. Then,
\[
U(t,x)=\lim_{M\to\infty}\lim_{n\to\infty}U_n^M(t,x).
\]
\end{theorem}

\begin{proof}
The result follows immediately from Lemma~\ref{lem:X^n to X} and the monotone convergence theorem.
\end{proof}

\begin{remark}
If $g$ is additionally bounded from above in Theorem~\ref{thm:U as a limit}, then we have
\[
U(t,x)= \lim_{n\to\infty}U_n(t,x).
\]
\end{remark}

\begin{remark}\label{rem:U as a limit}
Assume additionally that $\sigma$ satisfies \eqref{DS02} in Theorem~\ref{thm:U as a limit}. Since $\sigma_n\uparrow\sigma$, $\sigma_n$ must also satisfy \eqref{DS02}. Then, thanks to the local H\"{o}lder continuity of $\sigma_n$, we see from Theorem~\ref{thm:exist and unique} that $U_n^M$ is the unique classical solution to \eqref{PDE X g} with $\sigma$ and $g$ replaced by $\sigma_n$ and $g\wedge M$ respectively.

In other words, under \eqref{DS02}, Theorem~\ref{thm:U as a limit} characterizes the stochastic solution $U$ to the Cauchy problem \eqref{PDE X g} (which may not be smooth) as a limit of the unique classical solutions to some approximating Cauchy problems.
\end{remark}





\section{Summary of Results}
For readers' convenience, we summarize the main results of this paper in the following list, under different regularity assumptions on $\sigma$.

\begin{itemize}
  \item[(A)] If $\sigma$ satisfies \eqref{standing asm'}, then the following are equivalent:
			\begin{itemize}[leftmargin=0.1in]
			\item $X$ is a martingale; 
			\item	$\sigma$ satisfies \eqref{DS02};
			\item for any $T>0$ and $g\in\hat D$, $U(t,x)$ is the unique local stochastic solution to \eqref{PDE X g} in $\hat D$.
			\end{itemize}
  \item[(B)]  If $\sigma$ satisfies \eqref{standing asm'} and is locally H\"older continuous with exponent $\delta\in (0, 1]$, then the following are equivalent:
			\begin{itemize}[leftmargin=0.1in]
			\item	$\sigma$ satisfies \eqref{DS02};
			\item for any $T>0$ and $g\in\hat D$ bounded from below, $U(t,x)$ is the unique classical solution to \eqref{PDE X g} in $\hat D$.
			\item for any $T>0$ and $g\in\hat D$ bounded from below, a comparison theorem for \eqref{PDE X g} holds in $\hat D$.
			\end{itemize}	
 \item[(C)] If $\sigma$ satisfies \eqref{standing asm'} and \eqref{DS02}, and is continuous , then $U(t,x) $ can be characterized as a limit of the unique classical solutions to some approximating Cauchy problems.   
\end{itemize}

\noindent{\bf\large Acknowledgement.}

 We would like to thank Johannes Ruf for his comments, which directed us to several useful references. We also thank two anonymous reviewers whose careful reading and many suggestions helped us to improve  the paper significantly.


\appendix

\section{Consistent Extension of a Brownian Motion}
In this appendix, we assume that $X$ is a progressively measurable c\`adl\`ag stochastic process defined on some probability space
$(\Omega, \mathcal F, \mathbb P, \{\mathcal F_{t}\}_{t\ge 0})$ satisfying the usual condition. Given a stopping time $\tau$,
we will use the notation 
\[
X_{t\wedge \tau-} = \lim_{s\uparrow (t\wedge \tau)} X_{s},\ t\ge 0\quad \hbox{and}\quad X_{\cdot\wedge \tau-} = \{X_{t\wedge \tau-}\}_{t\ge 0}.
\]
\begin{lemma}\label{l:02}
 Suppose $\{\tau_{n}\}_{ n \in \N}$ is an increasing sequence of stopping times such that $X_{\cdot \wedge \tau_{n}}$ is a.s. continuous for all $n\in \mathbb N$. Then, $X_{ \cdot \wedge \tau_{\infty}-}$ is a.s. continuous, with $\tau_\infty := \lim_{n\to\infty} \tau_n$.
\end{lemma}

\begin{proof}
 Let $A_{n} = \{\omega: X_{ \cdot \wedge \tau_{n}} \hbox{ is continuous} \}$ and $A_{\infty} = \{\omega: X_{ \cdot \wedge \tau_{\infty}-} \hbox{ is continuous} \}$. By definition, $A_{n} \supseteq A_{n+1} \supseteq A_{\infty}$ for all $n$. Also, for any $\omega\in \bigcap_{n=1}^{\infty} A_{n}$,  $t\mapsto X_{t\wedge \tau_{n}}(\omega)$ is continuous for all $n$. This implies the continuity of $t\mapsto X_{t\wedge \tau_{\infty}-}(\omega)$, as demonstrated in the following three cases:
\begin{enumerate}
 \item If $t< \tau_{\infty}$, then there exists $\hat n\in\N$ such that $t< \tau_{\hat n}$. Consequently  $X_{\cdot \wedge \tau_{\infty}-} = X_{\cdot \wedge \tau_{\hat n}}$ in a neighborhood of $t$. Hence, $X_{\wedge \tau_{\infty}-}$ is continuous at $t$.
\item If $\tau_{\infty} < t < \infty$, then $X_{\cdot\wedge \tau_{\infty}-}$ is continuous at $t$ simply because it is a constant function around $t$; specifically, $X_{r\wedge \tau_\infty -} = \lim_{s\uparrow\tau_\infty} X_s$ for all $r>\tau_\infty$. 
\item If $t= \tau_{\infty}< \infty$, then $X_{\cdot \wedge \tau_{\infty}-}$ is left continuous at $t$ by definition: $X_{t\wedge \tau_{\infty}-} = \lim_{s\uparrow \tau_{\infty}} X_{s} = X_{\tau_\infty -} = \lim_{s\uparrow t} X_{s\wedge \tau_{\infty}-}$. Right continuity of $X_{\cdot \wedge \tau_{\infty}-}$ at $t$ also holds, since $X_{\cdot \wedge \tau_{\infty}-}$ is a constant function after $t=\tau_{\infty}$, as explained in 2.
\end{enumerate}
Thus, we conclude that $\bigcap_{n=1}^\infty A_{n} = A_{\infty}$. Since $\mathbb P(A_{n}) = 1$ for all $n$, we must have $\mathbb P(A_{\infty}) = 1$.
\end{proof}

\begin{remark}
In Lemma~\ref{l:02}, if the condition ``{\it $X_{\cdot \wedge \tau_{n}}$ is a.s. continuous for all $n\in \mathbb N$}'' is replaced by ``{\it  $X_{\cdot \wedge \tau_{n}}$ is a.s. non-explosive}'', then in general one can not conclude that ``{\it $X_{ \cdot \wedge \tau_{\infty}-}$ is a.s. non-explosive}''. A counter-example is $X$ under $\mathbb Q_{\infty}$ in the proof of Theorem~\ref{thm:X true mart}.
\end{remark}

\begin{lemma}\label{l:01}
 Let $W$ be a progressively measurable process on $(\Omega, \mathcal F, \mathbb P, \{\mathcal F_{t}\}_{t\ge 0})$. Suppose $\{\tau_{n}\}_{n \in \N}$ is an increasing sequence of stopping times such that $W_{\cdot \wedge \tau_{n}}$ is a standard Brownian motion for all $n\in\N$. Then, $W_{ \cdot \wedge \tau_{\infty} -}$ is a standard Brownian motion, with $\tau_\infty := \lim_{n\to\infty} \tau_n$.
\end{lemma}

\begin{proof}
By Lemma~\ref{l:02}, $W_{\cdot\wedge \tau_{\infty}-}$ and $W^{2}_{\cdot\wedge \tau_{\infty}-} - (\cdot\wedge\tau_{\infty})$ are both a.s.  continuous. For each $n$,  $W_{\cdot\wedge \tau_{n}}$ and $W^{2}_{\cdot\wedge \tau_{n}} - (\cdot\wedge\tau_{n})$ are martingales, and thus $W_{\cdot\wedge \tau_{\infty}-}$ and $W^{2}_{\cdot\wedge \tau_{\infty}-} - (\cdot\wedge\tau_{\infty})$ are local martingales. Levy's characterization theorem then ensures that $W_{\cdot\wedge \tau_{\infty}-}$ is a Brownian motion.
\end{proof}

Lemma \ref{l:01} states that a Brownian motion consistently extended by a sequence of increasing stopping times $\tau_n$, with a limit $\tau_\infty$, determines a Brownian motion on $[0, \tau_{\infty})$; such an extension, however, may not hold up to $\tau_\infty$. In other words, in general  we  can not conclude that $W_{ \cdot \wedge \tau_{\infty}}$ is a Brownian motion.  To see this, let $B$ be a standard Brownian motion and put
$\tau_{n} := \inf\{t>0: |B_{t}|\ge n\}\wedge (1-\frac{1}{n})$ and $\bar B_{t} := B_{t} I_{\{t < \tau_{\infty}\}}$ for $t \ge 0$. Then $\tau_{n}\uparrow \tau_{\infty} =1$ a.s. and the process  $\bar B_{}$ satisfies the condition in Lemma~\ref{l:01}. But, $\bar B_{\cdot\wedge \tau_\infty}$ has a discontinuity at $\tau_{\infty}$: 
\[
\lim_{t\uparrow \tau_\infty} \bar B_{t\wedge \tau_\infty} = \lim_{t\uparrow \tau_\infty} B_{t} = B_1 \neq 0 = \bar B_{\tau_{\infty}\wedge \tau_\infty}\quad  \hbox{a.s.}
\] 
Likewise, we can not generalize Lemmas~\ref{l:02} in a similar fashion.


\section{H\"{o}lder Norms Used in Theorem~\ref{thm:E[] is classical}}\label{sec:Holder norms}
Given $D\subseteq[0,\infty)\times\R^n$ and $f: D\mapsto\R$, we set $|f|_0:=\sup_D|f|$. For each $k\in\N$ and $\delta\in(0,1]$, by writing $X=(t,x)$ and $Y=(s,y)$, we define the following H\"{o}lder (semi-)norms:
\begin{align*}
[f]_{k+\delta,D}&:=\sup_{X\neq Y\text{ in } D}\sum_{\beta+2j=k}\frac{|D^\beta_x D^j_t f(X)-D^\beta_x D^j_t f(Y)|}{|X-Y|^\delta},\\
\langle f\rangle_{k+\delta,D}&:=\sup_{X\neq Y\text{ in } D,\text{ }x=y}\sum_{\beta+2j=k-1}\frac{|D^\beta_x D^j_t f(X)-D^\beta_x D^j_t f(Y)|}{|t-s|^{\delta+1}},\\
|f|_{k+\delta,D}&:=\sum_{\beta+2j\le k}|D^\beta_x D^j_t f|_0 + [f]_{k+\delta,D}+\langle f\rangle_{k+\delta,D}.
\end{align*}
We will denote by $H_{k+\delta}(D)$ the collection of functions $f$ with $|f|_{k+\delta,D}<\infty$.

Let $\partial^*D$ denote the parabolic boundary of $D$. For each $X=(t,x)\in D$, define
\[d(X):=\inf\{|X-(t',x')|:x'\in\partial^*D, t'>t\}.
\]
And we set $d(X,Y):=\min\{d(X),d(Y)\}$ for all $X,Y\in D$. Then, we define the following weighted H\"{o}lder (semi-)norms:
\begin{align*}
[f]^*_{k+\delta,D}&:=\sup_{X\neq Y\text{ in } D}\sum_{\beta+2j=k}d(X,Y)^{k+\delta}\frac{|D^\beta_x D^j_t f(X)-D^\beta_x D^j_t f(Y)|}{|X-Y|^\delta},\\
\langle f\rangle^*_{k+\delta,D}&:=\sup_{X\neq Y\text{ in } D,\text{ }x=y}\sum_{\beta+2j=k-1}d(X,Y)^{k+\delta}\frac{|D^\beta_x D^j_t f(X)-D^\beta_x D^j_t f(Y)|}{|t-s|^{\delta+1}},\\
|f|^*_{k+\delta,D}&:=\sum_{\beta+2j\le k}|D^\beta_x D^j_t f|^{\beta+2j}_0 + [f]^*_{k+\delta,D}+\langle f\rangle^*_{k+\delta,D}.
\end{align*}
We will denote by $H^*_{k+\delta}(D)$ the collection of functions $f$ with $|f|^*_{k+\delta,D}<\infty$.


\bibliographystyle{apt}


\begin{thebibliography}{10}

\bibitem{AGY80}
{\sc Az{\'e}ma, J. and Gundy, R. F. and Yor, M.} (1980).
\newblock Sur l'int\'egrabilit\'e uniforme des martingales continues.
\newblock In {\em Seminar on {P}robability, {XIV} ({P}aris, 1978/1979)
              ({F}rench)}.
\newblock vol.~784 of {\em Lecture Notes in Math.} Springer,
  Berlin,~53--61.

\bibitem{BKX12}
{\sc Bayraktar, E., Kardaras, C. and Xing, H.} (2012).
\newblock Valuation equations for stochastic volatility models.
\newblock {\em SIAM Journal on Financial Mathematics\/} {\bf 3,} 351--373.

\bibitem{BS12}
{\sc Bayraktar, E. and S{\^{\i}}rbu, M.} (2012).
\newblock Stochastic {P}erron's method and verification without smoothness
  using viscosity comparison: the linear case.
\newblock {\em Proc. Amer. Math. Soc.\/} {\bf 140,} 3645--3654.

\bibitem{BX10}
{\sc Bayraktar, E. and Xing, H.} (2010).
\newblock On the uniqueness of classical solutions of cauchy problems.
\newblock {\em Proceedings of the American Mathematical Society\/} {\bf 138
  (6),} 2061--2064.


\bibitem {CH05}
{\sc Cox, Alexander M. G. and Hobson, David G.} (2005).
\newblock Local martingales, bubbles and option prices.
\newblock {\em Finance Stoch.\/} {\bf 9,} 477--492.

\bibitem{DS02}
{\sc Delbaen, F. and Shirakawa, H.} (2002).
\newblock No arbitrage condition for positive diffusion price processes.
\newblock {\em Asia-Pacific Financial Markets\/} {\bf 9,} 159--168.

\bibitem{ET09}
{\sc Ekstr{\"o}m, E. and Tysk, J.} (2009).
\newblock Bubbles, convexity and the {B}lack-{S}choles equation.
\newblock {\em Ann. Appl. Probab.\/} {\bf 19,} 1369--1384.

\bibitem{ELY99}
{\sc Elworthy, K. D. and Li, Xue-Mei and Yor, M.} (2009).
\newblock The importance of strictly local martingales; applications to
              radial {O}rnstein-{U}hlenbeck processes.
\newblock {\em Probab. Theory Related Fields\/} {\bf 115,} 325--355.

\bibitem{ES85}
{\sc Engelbert, H.~J. and Schmidt, W.} (1985).
\newblock On one-dimensional stochastic differential equations with generalized
  drift.
\newblock In {\em Stochastic differential systems ({M}arseille-{L}uminy,
  1984)}.
\newblock vol.~69 of {\em Lecture Notes in Control and Inform. Sci.} Springer,
  Berlin,~143--155.

\bibitem{ES91}
{\sc Engelbert, H.~J. and Schmidt, W.} (1991).
\newblock Strong {M}arkov continuous local martingales and solutions of
  one-dimensional stochastic differential equations. {III}.
\newblock {\em Math. Nachr.\/} {\bf 151,} 149--197.

\bibitem{Friedman-book-06}
{\sc Friedman, A.} (2006).
\newblock {\em Stochastic differential equations and applications}.
\newblock Dover Publications Inc., Mineola, NY.
\newblock Two volumes bound as one, Reprint of the 1975 and 1976 original
  published in two volumes.

\bibitem{GS92}
{\sc Gal, S.~G. and Szabados, J.} (1992).
\newblock On monotone and doubly monotone polynomial approximation.
\newblock {\em Acta Math. Hungar.\/} {\bf 59,} 395--399.

\bibitem{Hsu85}
{\sc Hsu, P.} (1985).
\newblock Probabilistic approach to the {N}eumann problem.
\newblock {\em Comm. Pure Appl. Math.\/} {\bf 38,} 445--472.

\bibitem{HP08}
{\sc Hulley, H. and Platen, E.} (2008).
\newblock A visual classification of local martingales.
\newblock {\em Technical report}.
\newblock University of Technology Sydney.
\newblock Available at
  http://www.qfrc.uts.edu.au/research/research\textunderscore papers/rp238.pdf.

\bibitem{JT06}
{\sc Janson, S. and Tysk, J.} (2006).
\newblock Feynman-{K}ac formulas for {B}lack-{S}choles-type operators.
\newblock {\em Bulletin of the London Mathematical Society\/} {\bf 38,}
  268--282.

\bibitem{Karatzas-Ruf-13}
{\sc Karatzas, I. and Ruf, J.} (2015).
\newblock Distribution of the time to explosion for one-dimensional diffusions.
\newblock {\em Probability Theory and Related Fields\/} 1--43.

\bibitem{KS-book-91}
{\sc Karatzas, I. and Shreve, S.~E.} (1991).
\newblock {\em Brownian motion and stochastic calculus} second~ed. vol.~113 of
  {\em Graduate Texts in Mathematics}.
\newblock Springer-Verlag, New York.

\bibitem{KarlinT}
{\sc Karlin, S. and Taylor, H.~M.} (1981).
\newblock {\em A second course in stochastic processes}.
\newblock Academic Press, Inc. [Harcourt Brace Jovanovich, Publishers], New
  York-London.

\bibitem{Kotani06}
{\sc Kotani, S.} (2006).
\newblock On a condition that one-dimensional diffusion processes are
  martingales.
\newblock {\em Lecture Notes in Math.\/} {\bf 1874,} 149--156.

\bibitem{Lieberman92}
{\sc Lieberman, G.~M.} (1992).
\newblock Intermediate {S}chauder theory for second order parabolic equations.
  {IV}. {T}ime irregularity and regularity.
\newblock {\em Differential Integral Equations\/} {\bf 5,} 1219--1236.

\bibitem{Lib96}
{\sc Lieberman, G.~M.} (1996).
\newblock {\em Second order parabolic differential equations}.
\newblock World Scientific Publishing Co. Inc., River Edge, NJ.

\bibitem{MU12}
{\sc Mijatovi\'{c}, A. and Urusov, M.} (2012).
\newblock On the martingale property of certain local martingales.
\newblock {\em Probability Theory Relat. Fields\/} {\bf 152,} 1--30.

\bibitem{Pham-book-09}
{\sc Pham, H.} (2009).
\newblock {\em Continuous-time stochastic control and optimization with
  financial applications} vol.~61 of {\em Stochastic Modelling and Applied
  Probability}.
\newblock Springer-Verlag, Berlin.

\bibitem{RY99}
{\sc Revuz, D. and Yor, M.} (1999).
\newblock {\em Continuous martingales and {B}rownian motion} third~ed. vol.~293
  of {\em Grundlehren der Mathematischen Wissenschaften [Fundamental Principles
  of Mathematical Sciences]}.
\newblock Springer-Verlag, Berlin.

\bibitem{SV72}
{\sc Stroock, D. and Varadhan, S. R.~S.} (1972).
\newblock On degenerate elliptic-parabolic operators of second order and their
  associated diffusions.
\newblock {\em Comm. Pure Appl. Math.\/} {\bf 25,} 651--713.

\bibitem{Stroock-Varadhan-book-06}
{\sc Stroock, D.~W. and Varadhan, S. R.~S.} (2006).
\newblock {\em Multidimensional diffusion processes}.
\newblock Classics in Mathematics. Springer-Verlag, Berlin.
\newblock Reprint of the 1997 edition.

\end{thebibliography}

\end{document}